\documentclass[aap,reqno]{imsart}

\RequirePackage{amsthm,amsmath,amsfonts,amssymb}
\RequirePackage[authoryear]{natbib}
\RequirePackage[colorlinks,citecolor=blue,urlcolor=blue]{hyperref}
\RequirePackage{graphicx}
\usepackage{xcolor}
\usepackage{float}
\usepackage{subcaption}
\usepackage{caption}
\usepackage{dsfont}

\startlocaldefs
\theoremstyle{plain}
\newtheorem{theo}{Theorem}
\newtheorem{lemm}{Lemma}
\newtheorem{coro}{Corollary}
\newtheorem{prop}{Proposition}
\theoremstyle{definition}
\newtheorem{defi}{Definition}
\newtheorem{assu}{Assumption}

\theoremstyle{remark} 
\newtheorem{exam}{Example}
\newtheorem{rema}{Remark}
\renewcommand{\P}{{\mathbb P}}
\renewcommand{\d}{{\mathrm d}}
\newcommand{\E}{{\mathbb E}}
\newcommand{\R}{{\mathbb R}}
\newcommand{\N}{\mathbb N}

\newcommand{\as}{\qquad a.s.}
\newcommand{\I}[1]{\mathds{1}_{\{#1\}}}

\definecolor{R}{RGB}{255, 150, 0}
\definecolor{T}{RGB}{0, 100, 255}
\endlocaldefs
\begin{document}

\begin{frontmatter}
\title{On the perimeter estimation of pixelated excursion sets of 2D anisotropic random fields}
\runtitle{Perimeter estimation of excursion sets}

\begin{aug}
\author{\fnms{Ryan} \snm{Cotsakis}$^*$},
\author{\fnms{Elena} \snm{Di Bernardino}$^*$}
\and
\author{\fnms{Thomas} \snm{Opitz}$^\dagger$}
\address{$^*$Laboratoire J.A. Dieudonn\'e, Universit\'e C\^ote d’Azur, Nice, France}
\address{$^\dagger$Biostatistics and Spatial Processes, INRAE, Avignon, France}
\end{aug}

\begin{abstract}
We are interested in creating statistical methods to provide informative summaries of random fields through the geometry of their excursion sets.
To this end, we introduce an estimator for the length of the  perimeter of excursion sets of random fields on $\R^2$ observed over regular square tilings. The proposed estimator acts on the empirically accessible binary digital images of the excursion regions and computes the length of a piecewise linear approximation of the excursion boundary.
The estimator is shown to be consistent as the pixel size decreases, without the need of any normalization constant, and with neither assumption of Gaussianity nor isotropy imposed on the underlying random field. In this general framework, even when the domain grows to cover $\R^2$, the estimation error is shown to be of smaller order than the side length of the domain.
For affine, strongly mixing random fields, this translates to a multivariate Central Limit Theorem for our estimator when multiple levels are considered simultaneously. Finally, we conduct several numerical studies to investigate statistical properties of the proposed estimator in the finite-sample data setting.
\end{abstract}
%


\end{frontmatter}

\section{Introduction}

Random fields play a central role in the study of several real-world phenomena.
In many applications, the excursion set of a random field (\textit{i.e.}, the subset of the observation domain on which the random field exceeds a certain threshold) is observed---or partially observed---and its geometry can be used to make meaningful inferences about the underlying field.
Such techniques have been used in disciplines such as astrophysics \citep{gott1990,ade-planck2016}, brain imaging \citep{worsley1992}, and environmental sciences \citep{angulo2010, lhotka2015, frolicher2018}.
In certain cases, for example in landscape ecology, land-use analysis, and statistical modeling, understanding the geometry of excursions is of primary importance \citep{mcgarigal1995,nagendra2004, bolin2015}.\smallskip

\textit{Lipschitz-Killing curvatures} (abbreviated LKCs; also known as \textit{intrinsic volumes}) form a rich, well-known class of geometric summaries of stratified manifolds. \textit{Hadwiger's characterization theorem} states that LKCs form a basis for all rigid motion invariant valuations of convex bodies, which makes them central in the study of the geometry of random sets \citep{schneider2008}. From a theoretical point of view, probabilistic and statistical properties of the LKCs of excursion sets have been widely studied in the last decades \citep{adler2007}.
For Gaussian random fields, \textit{the Euler-Poincar\'e characteristic} (a well-studied, topological LKC) is studied in \citet{EL16} and \citet{dibernardino2017}; the excursion volume (another LKC, better known as the \textit{sojourn time} for one-dimensional processes) is studied in \citet{bulinski2012} and \citet{pham}. The reader is also referred to \cite{mueller2017} and \cite{kratz2018} for a joint analysis of LKCs and to \cite{meschenmoser2013} and \cite{shashkin2013} for functional central limit theorems.

LKCs have recently been used to create several statistical procedures including parametric inference \citep{bierme2019, dibernardino2020} and tests of Gaussianity \citep{dibernardino2017}, isotropy \citep{cabana1987, Fou17, berzin2021}, and  symmetry of  marginal distributions the underlying fields \citep{abaach2021}. \cite{Rossi} quantifies perturbation via the LKCs and provides a quantitative non-Gaussian limit theorem of the perturbed excursion area behaviour.
To further emphasize their importance, LKCs of excursions have deep links to extreme value theory; these insights are summarized in \cite{adler2007} and \cite{azais2009}. LKCs can thus provide meaningful and parsimonious summaries of the spatial properties of the studied random fields.\smallskip

In this manuscript, we focus on the two-dimensional setting---specifically, random fields defined on $\R^2$ endowed with the standard Euclidean metric. In this case, there are exactly three LKCs that can be leveraged to describe excursion sets of random fields in $\R^2$: the excursion volume (\textit{i.e.}, the area), half the value of the perimeter of the excursion set, and the Euler-Poincar\'e characteristic (which is equal to the number of connected components minus the number of holes of the excursion set).\smallskip

Analyzed jointly with information on the area and Euler characteristic of an excursion set, the  perimeter   provides valuable information about the fragmentation of the excursion set. Examples can be found in medical imaging where certain diseases can change fragmentation patterns in biological tissues \citep{Yao, Jurdi2021ASE}, or in ecology where suitable habitats of species are often characterized by exceedances of variables describing favorable conditions, and where edge effects near the boundary the excursion sets play an important role \citep{Debinski, Taubert}.  In spatial risk analysis, the perimeter can give information about the length of the interface between a high-risk zone (associated with exceedances of the threshold level) and moderate-to-low risk zones.
    
Most of the results presented in the previous literature are based on the empirically inaccessible knowledge of the continuous random field $X$ on a compact domain $T \subset \R^2$. In practice, spatial data are  often observed only at sampling locations on a discrete grid $\{s_{i,j}:i,j\in\N_0\}\cap T$, and in such cases, the values of the random field at intermediate points between the sampling locations are not empirically accessible.  This regular lattice setting is popular, for example, in the areas
of remote sensing, computer vision, biomedical imaging, surface meteorology.   The datum at the sampling location $s_{i,j}$ could conceivably be a floating point number representing the value of the random field at $s_{i,j}$, however, it may be the case that this level of precision is not available.
One can also consider the  more general case where the accessible information at the sampling location $s_{i,j}$ is a boolean value corresponding to whether the random field evaluated at $s_{i,j}$ falls within a predetermined interval---normally $[u,\infty)$ for fixed $u\in\R$. In this general case, one obtains a pixelated representation of the excursion set of $X$ at the fixed level~$u$.\smallskip

From these sparse-information, binary digital images of excursion sets, we aim in the present work to infer the second Lipschitz-Killing curvature, \emph{i.e.}, the perimeter of the excursion set, for a fixed level $u$. The perimeter is a particularly difficult quantity to estimate, since, in a digital image, the boundary of an object is comprised of vertical and horizontal pixel edges, which obviously does not correspond to the object's true boundary. There exists a number of algorithms for computing the perimeter of objects in hard segmented (\textit{i.e.} binary) digital images, many of which are summarized in \cite{coeurjolly2004} with further developments made in \cite{deVieilleville2007}. It seems, however, intractable to evaluate the performance of these algorithms on excursion sets of two-dimensional random fields.  \cite{bierme2021} studies how the integrated perimeter of excursion sets over a set of levels changes when considering discretized versions of the underlying \textit{stationary}, \textit{isotropic} random fields (\textit{i.e.}, those with translation- and rotation-invariant distributions). This gives rise to a perimeter estimator for a single level, complete with its own probabilistic analysis for isotropic random fields \citep{bierme2021}. The estimator is further analyzed and given explicit covariance formulas in \cite{abaach2021} for the case of complete spatial independence. Although this particular perimeter estimator is quite natural to study, it suffers from certain defects; namely, an intrinsic inadequacy for anisotropic random fields.\smallskip

We introduce a class of estimators for the perimeter of objects in binary digital images, one of which being particularly suitable for estimating the perimeter of excursion sets of anisotropic random fields on $\R^2$.
The elements of the class are uniquely associated to the choice of norm that is used to measure a piecewise linear approximation of the excursion's boundary. The estimator derived from the work of \cite{bierme2021} arises as the element of the proposed class associated to the 1-norm. The novel estimator associated to the 2-norm (the primary focus of this paper) possesses the desirable property of \textit{multigrid convergence} (\textit{i.e.}, strong consistency as the pixel size tends to zero; see Theorem~\ref{thm:consistent_Phat}), which we extend to convergence in mean (see Proposition~\ref{prp:L1}).
These general results hold under weak assumptions about the  smoothness of the random field that do not include Gaussianity, nor isotropy. As the domain grows to cover $\R^2$, sufficient conditions are given such that the error in the estimation is of smaller order than the fluctuations of the perimeter---making the limiting distributions of the perimeter and the estimator identical. In particular, by further supposing that the underlying random field is \textit{affine} and \textit{strongly mixing} (notions described in Section~\ref{sec:asymptotic_normality}), the estimator associated to the 2-norm is asymptotically normal with the same asymptotic variance as perimeter itself (see Theorem~\ref{thm:CLT}).\smallskip

The organization of the paper is as follows. Section~\ref{sec:defs} specifies key notions including: excursion sets, the hypotheses on the underlying random fields, the regular grid on which the excursion sets are observed, and the novel class of considered perimeter estimators.
In Section~\ref{sec:main_results}, the statistical properties of the perimeter estimate based on the 2-norm are discussed for a fixed domain (Section~\ref{sec:results_fixedT}) and for a sequence of growing domains (Section~\ref{sec:results_growingT}).
Section~\ref{sec:sims} provides extensive numerical results to support and illustrate the theory developed in Section~\ref{sec:main_results}.
Proofs and auxiliary notions are postponed to Section~\ref{sec:proofs}. We conclude with a discussion section. Some supplementary elements are provided in the Appendix Section.

\section{Definitions and Notation}\label{sec:defs}

Let us begin by introducing some notation. Calligraphic font is used to denote sets of isolated points in $\R^2$.
For a set $S\subset \R^2$, its boundary is denoted $\partial (S)$; its cardinality $\#(S)$; and its Lebesgue measure $\nu(S)$. We use $\mathcal{H}^1$ to denote the one-dimensional Hausdorff measure, and $C^k$ to denote the space of real-valued functions on $\R^2$ with $k$ continuous derivatives.
Between the nomenclatures \textit{sample paths} and \textit{trajectories}, we choose to use the former when describing the realizations of a random field.

The following assumption ensures that the random objects that we consider are well defined.
\begin{assu}\label{ass:basic}
The real-valued random field $X= \{X(s): s\in\R^2\}$ defined on a probability space $(\Omega, \mathcal{F}, \P)$ has $C^2$ sample paths.
\end{assu}

\begin{defi}\label{def:excursion}
    Denote the \emph{excursion set} of $X$ at the level $u\in\R$ by
        $E_{X}(u) := \{s\in\R^2:X(s) \geq u\}$.
    For compact $T\subset \R^2$, we denote the restriction of $E_X(u)$ and $\partial\big( E_X(u)\big)$ to $T$ by
    \begin{equation*}
    E_X(T,u) := T\cap E_X(u)\qquad \mathrm{and}\qquad E^{\partial}_X(T,u) := T\cap \partial\big( E_X(u)\big)
    \end{equation*}
    respectively. Finally, the quantity of interest in this paper:
    \begin{equation*}
        P_{X}^{T}(u) := \mathcal{H}^1\big(E_X^\partial(T,u)\big).
    \end{equation*}
\end{defi}

In Figure~\ref{fig:excursion_continuous} (a), a $C^2$ sample path of a Gaussian random field $X$ is depicted in a square domain $T$ with the contours $E^\partial_X(T,u)$ drawn on the domain for various levels $u$. In Figure~\ref{fig:excursion_continuous} (b) and (c), $E_X(u)$ is represented by the dark regions, for two different levels $u$.

\begin{figure}
     \centering
     \begin{subfigure}[b]{0.35\linewidth}
         \centering
        \includegraphics[width=\linewidth]{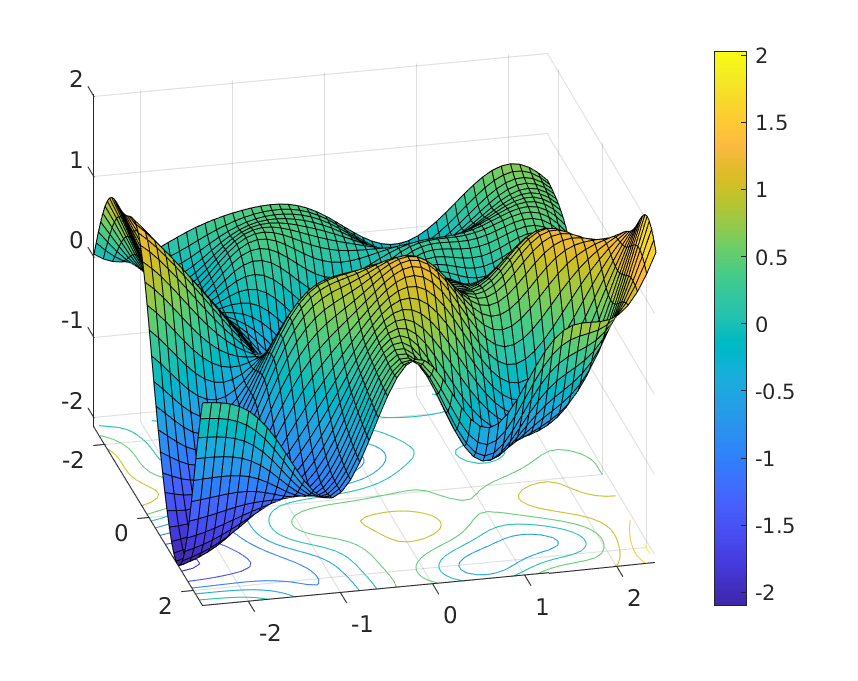}
         \caption{}
     \end{subfigure}
     \begin{subfigure}[b]{0.50\linewidth}
         \centering
         \includegraphics[width=\linewidth]{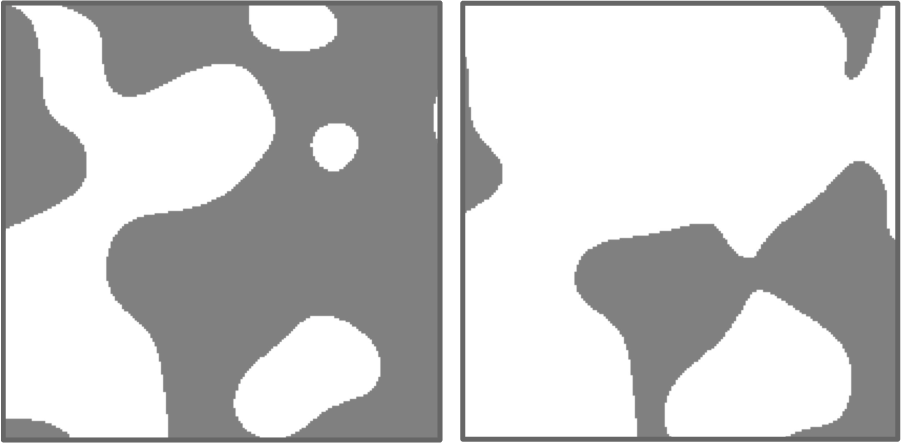}
         \caption{\qquad\qquad\qquad\qquad\ (c)}
     \end{subfigure}
     \vspace{-0.3cm}
        \caption{Panel (a): a $C^2$ realization of a stationary, centered, Gaussian random field $X$ with covariance function $r_X(h) = \exp(-||h||_2^2)$ is depicted in the square-shaped observation window $T=[-2.5,2.5]^2$ (generated using the R package \texttt{RandomFields} \citep{RandomFields}). Underneath the sample path, the curves $E^\partial_X(T,u)$ are drawn for different values of $u$. Panel (b) (\emph{resp.} panel (c)): the dark region $E_X(T,u)$ is shown for $u=0$ (\emph{resp.} $u=0.5$).}
        \label{fig:excursion_continuous}
\end{figure}

\smallskip

In what follows, let
\begin{equation}\label{eqn:def_T}
T := [-t,t]^2\subset \R^2,
\end{equation}
for fixed $t > 0$.
Before proceeding, it is helpful to specify additional assumptions on the considered random fields.

\begin{assu}\label{ass:weak}
Let $X_1$ and $X_2$ denote the partial derivatives of $X$ in the two principle Cartesian directions in $\R^2$, and let $X_{11}$ and $X_{22}$ denote the corresponding second order partials.
For any $u\in\R$, the following three conditions hold almost surely:
\begin{enumerate}
\item $X$ has no critical points in $T$ at the level $u$.
\item The restriction of $X$ to each face of the square boundary $\partial (T)$ has no local extrema at the level $u$.
\item For $k\in\{1,2\}$, there are no $s\in T$ such that $X(s)-u = X_k(s) = X_{kk}(s) = 0$.
\end{enumerate}
\end{assu}

Together, Assumptions \ref{ass:basic} and \ref{ass:weak} ensure that the random field $X$ is almost surely \textit{suitably regular} at the level $u$ in $T$ as defined in \citet[Definition~6.2.1]{adler2007}. The third condition of Assumption~\ref{ass:weak} is made to be slightly stronger than item (C) in Definition~6.2.1 of \cite{adler2007} so that the suitably regular condition holds even after a permutation of the two principal Cartesian directions.
This is useful when considering the set
\begin{equation}\label{eqn:curly_Y}
\mathcal{Y}_{X}^T(u) := \bigcup_{k=1,2}\{s\in E^\partial_X(T,u): X_k(s) = 0\}.
\end{equation}
Indeed, under Assumptions \ref{ass:basic} and \ref{ass:weak}, it follows directly from \citet[Lemma~6.2.3]{adler2007} that
\begin{equation}\label{eqn:gradient_ass_k}
    \#\big(\mathcal{Y}_{X}^T(u)\big)<\infty, \as
\end{equation}

Recall that the \textit{reach} of a set $S\subset \R^d$ is given by
\begin{equation}\label{eqn:reach}
\mathrm{reach}(S) := \sup\{\delta\geq 0: \forall y \in S_\delta\  \exists!x\in S\ \mathrm{nearest\ to}\ y\},
\end{equation}
where $S_\delta = \big\{y\in\R^d: \exists\ x \in S\ \mathrm{s.t.}\ ||x-y||_2\leq \delta \big\}$ is the dilation of the set $S$ by a radius $\delta \geq 0$ (see, \textit{e.g.}, Definition~11 in \cite{thale2008}). Equations~\eqref{eqn:gradient_ass_k} and~\eqref{eqn:reach} will be useful later (see, for example, Remark~\ref{rem:positive_as}).
\smallskip

Recall that a curve $\gamma\subset\R^2$ is connected if it cannot be expressed as the union of two disjoint nonempty closed sets in $\R^2$. For sets $B \subseteq A \subseteq \R^2$, $B$ is maximally connected in $A$ if $B$ is connected and there does not exist a connected $C\subseteq A$ such that $B\subset C$.
\begin{defi}
Let $\Gamma_X^T(u)$ be the set of maximally connected subsets of $E^\partial_X(T,u)$.
\end{defi}

\begin{assu}\label{ass:integrable_quantities}
The random variables $P_X^T(u)$ and $\#\big(\Gamma_X^T(u)\big)$ are in $L^1(\Omega)$, the space of integrable random variables, for all $u\in\R$.
\end{assu}

We emphasize that none of the assumptions stated thus far restrict to stationary or isotropic random fields. Although stationarity is assumed in Theorem~\ref{thm:CLT} and Corollary~\ref{cor:gaussian}, these results and all other results are applicable to anisotropic random fields---a crucial point that we investigate numerically in Section~\ref{sec:aniso}.\smallskip

In what follows, we study a novel estimator of the random quantity $P_{X}^T(u)$ for arbitrary but fixed $u\in\R$, based only on the random field $Z_X(\cdot; u) = \{Z_X(s;u):s\in\R^2\}$ defined by
\begin{equation*}
Z_X(s;u) := \I{s\in E_{X}(u)} = \I{X(s) \geq u},\qquad s\in\R^2.
\end{equation*}
Note that $Z_X(s;u)$ has dependent Bernoulli margins with parameter $\P\big(X(s)\geq u\big)$.
We will assume that $Z_X(\cdot; u)$ is empirically accessible only at sampling locations on a regular grid, one that is defined in Section~\ref{gridSection} below.

\subsection{Sampling locations on a regular grid}\label{gridSection}

\begin{defi}\label{def:grid}
Fix $\epsilon > 0$, and define a square grid of points in $\R^2$ as
\begin{equation}\label{eqn:grid_def}
\mathcal{G}^{(T,\epsilon)} := \big\{s_{i,j}: i,j\in\N_0\big\}\cap T, \,\,\, \mbox{ with } \,\, s_{i,j} := (-t+i\epsilon, -t+j\epsilon)\in\R^2,
\end{equation}

and with $T$ and $t$ as in equation~\eqref{eqn:def_T}.
Let $M$ be the number of rows (which is consequentially identical to the number of columns) of $\mathcal{G}^{(T,\epsilon)}$. Define the index set
$$I^{(T, \epsilon)} := \{0,\ldots,M-1\}\subset \N_0$$
and the random matrix $\zeta_X^{(T,\epsilon)}(u)$ with binary elements
\begin{equation}\label{eqn:zeta_def}
\zeta_{X,i,j}^{(T,\epsilon)}(u)  := Z_X(s_{i,j};u) =  \I{X(s_{i,j}) \geq u},
\end{equation}
for $ i,j\in I^{(T, \epsilon)}$. For $m\in \N^+$, let us define
$$ I^{(T, \epsilon, m)} := \{i\in I^{(T, \epsilon)}: i \equiv 0\ (\mathrm{mod}\ m)\}.$$
\end{defi}

Notice that
$\mathcal{G}^{(T,\epsilon)} = \{ s_{i,j} : i,j\in I^{(T, \epsilon)}\}.$
We provide an illustration of $\mathcal{G}^{(T,\epsilon)}$ in Figure~\ref{fig:sampling}, where the elements with indices in $ I^{(T,\epsilon,m)}$, with $m=2$, are highlighted in red. We highlight that our proposed estimator for $P_{X}^T(u)$ will be based only on the sparse observations $\zeta_{X,i,j}^{(T,\epsilon)}(u)$ for $i,j\in I^{(T,\epsilon)}$ (see Section \ref{sec:estimator}).

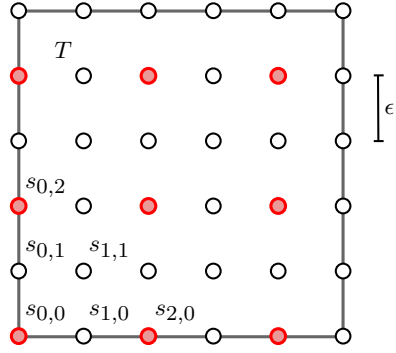
\begin{figure}
    \centering
    \includegraphics[width=0.35\linewidth]{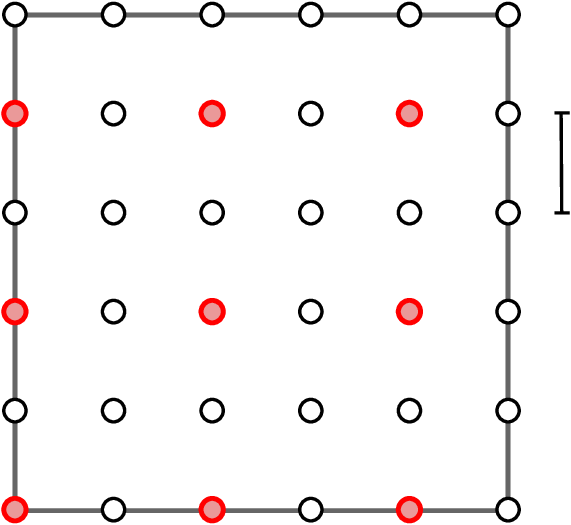}
    \put(0,88){$\epsilon$}
    \put(-125,109){$T$}
    \put(-136,12){$s_{0,0}$}
    \put(-136,36){$s_{0,1}$}
    \put(-136,60){$s_{0,2}$}
    \put(-112,12){$s_{1,0}$}
    \put(-112,36){$s_{1,1}$}
    \put(-87,12){$s_{2,0}$}
    \caption{An illustration of the quantities defined in Definition~\ref{def:grid}. The positions of the elements of $\mathcal{G}^{(T,\epsilon)}$ in $\R^2$ are shown as circles, and the subset $\{s_{i,j} : i, j \in  I^{(T,\epsilon,m)}\}$ with $m=2$ is highlighted in red. Here, $M=6$, and the side length of $T$ is $\sqrt{\nu(T)} = (M-1)\epsilon = 5\epsilon$.}
    \label{fig:sampling}
\end{figure}

\begin{rema}\label{rem:pixels}
The data matrix $\zeta_X^{(T,\epsilon)}(u)$ in~\eqref{eqn:zeta_def} can be represented as a binary digital image as depicted in Figure~\ref{fig:excursion_discrete} (b). In this framework, $M$ corresponds to the pixel density or grid size of the image (an integer number of pixels per distance of $2t$, the side length of $T$), and $\epsilon$ corresponds to the pixel width. The quantities are related by
$|M\epsilon - 2t|\leq \epsilon.$
\end{rema}

\begin{figure}
    \centering
    \includegraphics[width=0.57\linewidth]{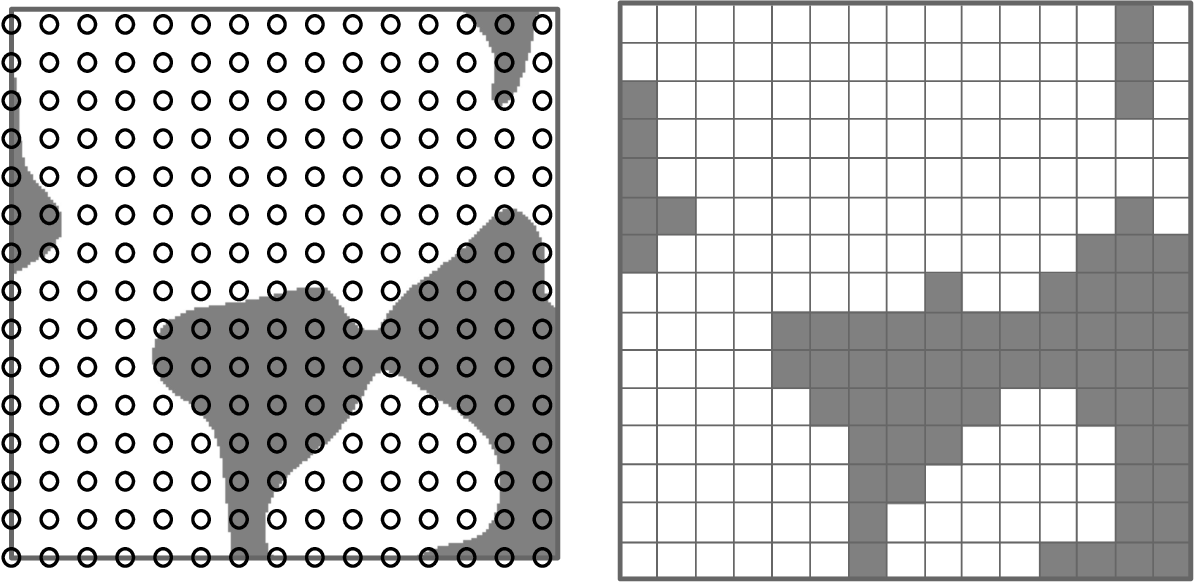}
    \put(-182, -10){(a)}
    \put(-62, -10){(b)}
    \caption{Panel  (a):  $E_X(T,0.5)$, as shown in Figure~\ref{fig:excursion_continuous} panel (c), superposed with the elements of the grid $\mathcal{G}^{(T,\epsilon)}$ shown as black circles. Here, $\epsilon\approx 0.32$. Panel  (b): the binary matrix $\zeta_X^{(T,\epsilon)}(0.5)$, defined in~\eqref{eqn:zeta_def}, represented as a binary digital image (dark pixels corresponding to 1, and white~to~0).}
    \label{fig:excursion_discrete}
\end{figure}

\subsection{Definition of the estimators}\label{sec:estimator}
Here, we introduce a class of estimators of $P_X^T(u)$ that use only the information contained in $\zeta_X^{(T,\epsilon)}(u)$, defined in~\eqref{eqn:zeta_def}.
Loosely speaking, $\zeta_X^{(T,\epsilon)}(u)$ is separated into submatrices, and in each submatrix the length of the line segment that approximately separates the 1's from the 0's is computed.
In this way, the estimator obtained depends on the choice of norm used.
\begin{defi}\label{def:estimator}
With $||\cdot||_{p}$ denoting the $p$-norm, for $p\in\N^+$, define
\begin{equation}\label{eqn:phat}
    \hat{P}_{X}^{(p)}(\epsilon, m; T, u) := \, \epsilon\sum_{a\in I^{(T, \epsilon, m)}}\sum_{b\in I^{(T, \epsilon, m)}} \big|\big|\big(N_{X,h}(a,b;u), N_{X,v}(a,b;u)\big)\big|\big|_p,
\end{equation}
where
\begin{equation*}
    N_{X,h}(a,b;u) := \sum_{i = a}^{(a+m-1) \wedge (M-1)}\ \sum_{j = b}^{(b+m-1) \wedge (M-2)} |\zeta_{X,i,j}^{(T,\epsilon)}(u) - \zeta_{X,i,j+1}^{(T,\epsilon)}(u)|,\ \ \ a,b\in I^{(T,\epsilon,m)},
\end{equation*}
and
\begin{equation*}
    N_{X,v}(a,b;u) := \sum_{i = a}^{(a+m-1) \wedge (M-2)}\ \sum_{j = b}^{(b+m-1) \wedge (M-1)} |\zeta_{X,i,j}^{(T,\epsilon)}(u) - \zeta_{X,i+1,j}^{(T,\epsilon)}(u)|,\ \ \ a,b\in I^{(T,\epsilon,m)}.
\end{equation*}
\end{defi}

Continuing from the framework discussed in Remark \ref{rem:pixels}, $N_{X,v}$ (\emph{resp.} $N_{X,h}$) counts the number of pixels in a subrectangle---of size at most $m\times m$ pixels---of $T$ that differ in shade from the neighbouring pixel to the right (\emph{resp.} above). In other words, $N_{X,v}$ (\emph{resp.} $N_{X,h}$) provides a count of significant vertical (\emph{resp.} horizontal) pixel edges in the subrectangle.

By considering the estimator in~\eqref{eqn:phat} with norm $p=1$, one recovers the estimator that is extensively studied in \cite{bierme2021} and \citet{abaach2021}. It counts the number of pixel edges that separate pixels of different color, and rescales the count by $\epsilon$. Thus, $\hat{P}^{(1)}_X(\epsilon, m; T,u)$ will not depend on $m$, so we write $\hat{P}_{X}^{(1)}(\epsilon; T, u)$ in place of $\hat{P}_{X}^{(1)}(\epsilon, m; T, u)$.\smallskip

Figure~\ref{fig:estimators} illustrates the behavior of the estimator in equation~\eqref{eqn:phat} constructed with two different norms; the norms associated to $p=1$ and $p=2$. In addition, Table~\ref{tab:computations} provides the corresponding terms in equation~\eqref{eqn:phat} for this example, for each $a,b\in I^{(T,\epsilon,2)} = \{0,2,4\}$, for both $p=1$ (second-last column) and $p=2$ (last column).

\begin{figure}
    \centering
    \includegraphics[width=0.76\linewidth]{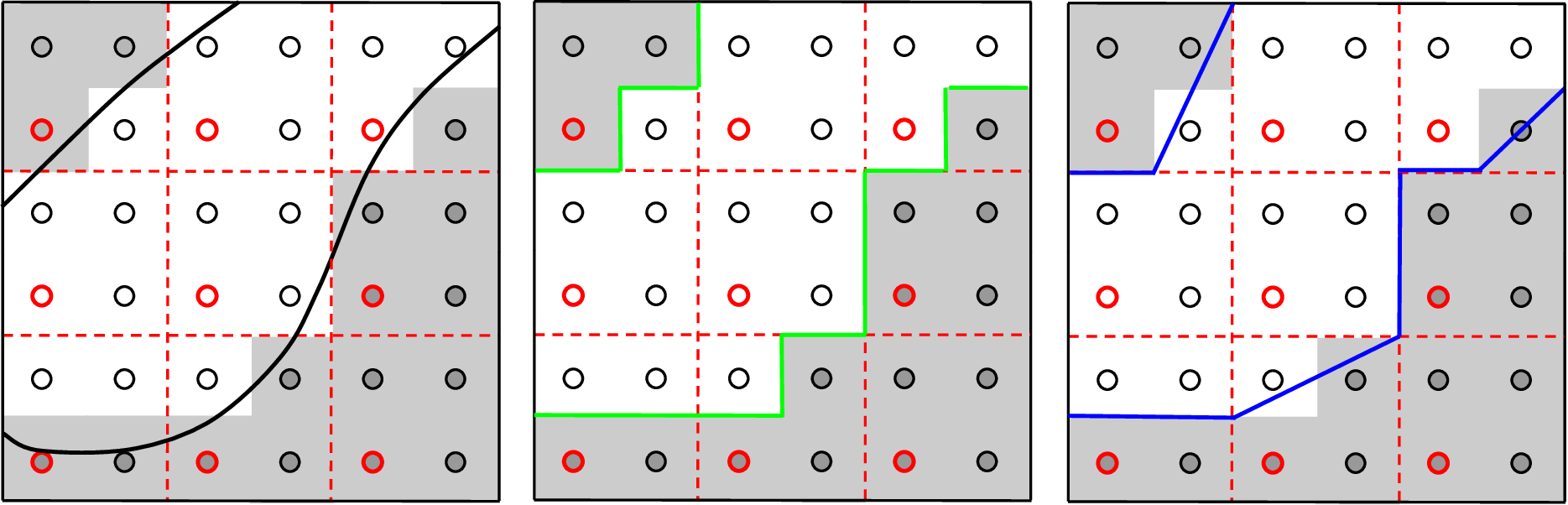}
    \put(-264, -10){(a)}
    \put(-162, -10){(b)}
    \put(-60, -10){(c)}
    \caption{Panel (a): the curve $E^\partial_X(T,u)$ is shown in relation to the points in $\mathcal{G}^{(T,\epsilon)}$ in \eqref{eqn:grid_def}. Points in the dark regions are assigned a value of 1 in the matrix $\zeta_X^{(T,\epsilon)}(u)$, and points in white are assigned a value of 0. The points outlined in red have indices in $ I^{(T,\epsilon,m)}$ with $m=2$. In effect, $\hat{P}_{X}^{(1)}(\epsilon; T, u)$ is calculated by counting the pixel edges shown in green (see panel (b)), whereas  $\hat{P}_{X}^{(2)}(\epsilon, 2; T, u)$ is calculated by summing the lengths of the blue piecewise linear curves  (see panel (c)).}
    \label{fig:estimators}
\end{figure}

\begin{table}
\caption{$\hat{P}_{X}^{(p)}(\epsilon, m; T, u)$ in~\eqref{eqn:phat} computed for the discretized excursion set in Figure~\ref{fig:estimators}. The last two columns correspond to the terms $\big(N_{X,v}(a,b;u)^p + N_{X,h}(a,b;u)^p\big)^{1/p}$ for $p=1$ and $p=2$. Summing each term yields $\hat{P}_{X}^{(p)}(\epsilon, m; T, u)$, as shown in bold in the final row.}
\centering
\begin{tabular}{ |p{1.75cm}|p{1.75cm}||p{1.75cm}|p{1.75cm}||p{1cm}|p{1cm}| }
 \hline
 $a\in I^{(T,\epsilon,m)}$& $b\in I^{(T,\epsilon,m)}$& $N_{X,h}(a,b;u)$& $N_{X,v}(a,b;u)$& $p=1$ & $p=2$\\
 (column) & (row)&&&&\\
 \hline
 0& 0& 2& 0& 2& 2\\
 0& 2& 1& 0& 1& 1\\
 0& 4& 1& 2& 3& $\sqrt{5}$\\
 2& 0& 2& 1& 3& $\sqrt{5}$\\
 2& 2& 0& 2& 2& 2\\
 2& 4& 0& 0& 0& 0\\
 4& 0& 0& 0& 0& 0\\
 4& 2& 1& 0& 1& 1\\
 4& 4& 1& 1& 2& $\sqrt{2}$\\
 \hline
 \multicolumn{4}{|c|}{}&&\\[-1em]
 \multicolumn{2}{|c}{} & \multicolumn{2}{c|}{$\hat{P}_{X}^{(p)}(\epsilon, m; T, u):$} & \textbf{14}$\boldsymbol\epsilon$ & \textbf{11.89}$\boldsymbol\epsilon$\\[3pt]
 \hline
\end{tabular}
\label{tab:computations}
\end{table}

The estimator in~\eqref{eqn:phat} with norm $p=2$ approximates the length of $E^\partial_X(T,u)$ by the total length of a set of line segments that approximate the curve (see Figure~\ref{fig:estimators} (c)). The number of possible orientations of each line segment grows with $m$; so does the length of each line segment, which, loosely speaking, is on the order of $m\epsilon$. Therefore, it is not surprising that $\hat{P}_{X}^{(2)}(\epsilon, m; T, u)$ depends on $m$, and our statistical analysis in Section~\ref{sec:main_results} therefore takes place in the regime where $m$ is large and $m\epsilon$ is small. In Section~\ref{sec:hyperparam}, we provide an adaptive method to select the hyperparameter $m$ when $\epsilon$ is given as a feature of the data.

\section{Main Results}\label{sec:main_results}

The focus of this section is to prove convergence results for the estimator $\hat{P}_{X}^{(2)}(\epsilon, m; T, u)$. The statistical analysis is separated into two regimes. In Section~\ref{sec:results_fixedT}, we consider the domain $T$ to be fixed and decrease the pixel width while sending $m$ to infinity. Section~\ref{sec:results_growingT} studies the behaviour of the estimator on a sequence of growing domains. In particular, in Section \ref{sec:results_growingT1}, we study the asymptotic relationships between $\epsilon$, $m$, and the Lebesgue measure of the sequence of domains, and provide sufficient conditions for good convergence properties. We conclude with a multivariate Central Limit Theorem in the case where multiple levels $(u_1,\ldots,u_k)$ are considered simultaneously under the assumption that the underlying random field $X$ is affine and strongly mixing (see Section~\ref{sec:asymptotic_normality} for the theorem and the notions of affinity and strongly mixing).

\subsection{On a fixed domain with decreasing pixel width}\label{sec:results_fixedT}

Here, we are interested in the behaviour of the estimator $\hat{P}_{X}^{(2)}(\epsilon, m; T, u)$ in the case where the domain $T=[-t,t]^2$ is fixed, and the spacing between the locations of the observations in the matrix $\zeta_X^{(T,\epsilon)}(u)$ tends to 0. We proceed to show that the resulting perimeter estimate converges almost surely to $P_X^T(u)$ and give the rate of convergence.

\begin{theo}\label{thm:consistent_Phat}
Let $(m_n)_{n\geq 1}$ be a non-decreasing sequence in $\N^+$ tending to $\infty$ as $n\rightarrow\infty$.
Let $(\epsilon_n)_{n\geq 1}$ be a sequence in $\R^+$ such that $m_n\epsilon_n^{2/3}$ converges to a constant $C\in\R^+$ and that the vertices of $T$ are contained in $\mathcal{G}^{(T,\epsilon_n)}$ for all $n\in \N^+$. Then, under Assumptions \ref{ass:basic} and \ref{ass:weak}, for fixed $u\in\R$, it holds that
$$g_n\big|\hat{P}^{(2)}_X(\epsilon_n, m_n; T, u) - {P}_{X}^T(u)\big|\stackrel{\mathrm{a.s.}}{\longrightarrow}0,\qquad n\rightarrow\infty,$$
where $(g_n)_{n\geq 1}$ is any non-decreasing sequence such that $g_n = o(m_n)$.
\end{theo}

The proof of Theorem~\ref{thm:consistent_Phat} is postponed to Section~\ref{sec:proofs}.

\begin{rema}
Theorem~\ref{thm:consistent_Phat} is a statement about the \textit{multigrid convergence} (see, for instance, Definition~2 of \cite{coeurjolly2004}) of $\hat{P}^{(2)}_X(\epsilon_n, m_n; T, u)$ to ${P}_{X}^T(u)$ as $n \rightarrow \infty$ for almost all sample paths of the random field $X$. The \textit{speed} of this convergence is $O(1/m_n)$.
\end{rema}

Theorem~\ref{thm:consistent_Phat} requires that the vertices of $T$ are in $\mathcal{G}^{(T,\epsilon_n)}$ for all $n\in\N^+$, for example as depicted in Figure~\ref{fig:sampling}. This prevents the possibility of there being long segments of $E^\partial_X(T,u)$ that remain close to the border of $T$ so as to not pass between elements of $\mathcal{G}^{(T,\epsilon_n)}$. In addition, it is supposed that the sequence $(m_n)_{n\geq 1}$ is asymptotically equivalent to $(\epsilon_n^{-2/3})_{n\geq 1}$, which gives the fastest possible rate of convergence of $\hat{P}^{(2)}_X(\epsilon_n, m_n; T, u)$ to ${P}_{X}^T(u)$. By relaxing this condition, we obtain the following corollary.

\begin{coro}\label{cor:as_convergence}
Under the conditions of Theorem~\ref{thm:consistent_Phat}, if the requirement that $m_n\epsilon_n^{2/3}\rightarrow C$ is relaxed to $m_n\epsilon_n\rightarrow 0$, it holds that
$$\hat{P}^{(2)}_X(\epsilon_n, m_n; T, u) \stackrel{\mathrm{a.s.}}{\longrightarrow} P_X^T(u),\qquad n\rightarrow\infty.$$
\end{coro}

The proof is postponed to Section~\ref{sec:proofs}. The following proposition shows that convergence in $L^1(\Omega)$ holds under slightly stronger assumptions. The proof can also be found in Section~\ref{sec:proofs}.

\begin{prop}\label{prp:L1}
Let $(m_n)_{n\geq 1}$ be a non-decreasing sequence in $\N^+$ tending to $\infty$ as $n\rightarrow\infty$.
Let $(\epsilon_n)_{n\geq 1}$ be a sequence in $\R^+$ such that $m_n\epsilon_n\rightarrow 0$ as $n\rightarrow\infty$, and that the vertices of $T$ are contained in $\mathcal{G}^{(T,\epsilon_n)}$ for all $n\in \N^+$. Then under Assumptions \ref{ass:basic}, \ref{ass:weak}, and \ref{ass:integrable_quantities},
$$\big|\hat{P}^{(2)}_X(\epsilon_n, m_n; T, u) - {P}_{X}^T(u)\big | \stackrel{L^1}{\longrightarrow} 0,\qquad n\rightarrow\infty,$$
for any fixed $u\in\R$.
\end{prop}

\begin{rema}
It is shown in Proposition~5 of \cite{bierme2021} that for a random field $X$ satisfying Assumption~\ref{ass:basic}, if, in addition, $X$ is stationary, Gaussian, isotropic, and the supremum of the first and second order partial derivatives of $X$ in the domain $T$ are in $L^1(\Omega)$, then
\begin{equation}\label{eqn:bierme_converge}
\E[\hat{P}_{X}^{(1)}(\epsilon; T, u)] \rightarrow \frac{4}{\pi}\E[P_X^T(u)],
\end{equation}
as $\epsilon \rightarrow 0$.
Proposition~\ref{prp:L1} is a stronger result under weaker assumptions on $X$. With neither Gaussianity, stationarity, nor isotropy imposed on $X$, it holds that
$$\E[\hat{P}_{X}^{(2)}(\epsilon,m; T, u)] \rightarrow \E[P_X^T(u)],$$
as $\epsilon \rightarrow 0$ and $m\rightarrow\infty$ under the constraint $m\epsilon \rightarrow 0$.
Thus, the estimator $\hat{P}_{X}^{(2)}(\epsilon, m; T, u)$ does not suffer from the asymptotic bias factor of $4/\pi$.
\end{rema}

\subsection{On a growing domain with decreasing pixel width}\label{sec:results_growingT}

In this section, the performance of $\hat{P}_{X}^{(2)}(\epsilon_n,m_n; T_n, u)$ is investigated for sequences $(\epsilon_n)_{n\geq 1}$, $(m_n)_{n\geq 1}$, and $(T_n)_{n\geq 1}$ satisfying $\epsilon_n\rightarrow 0$, $m_n\rightarrow\infty$, and $T_n \nearrow \R^2$ as $n\rightarrow\infty$. To manage the added complexity of the sequence of growing domains, first define
$$T_n:= \{ns:s\in T\},$$
such that $T_n$ is a dilation of the fixed domain $T=[-t,t]^2$. The side length of the square domain $T_n$ is then $2tn$. The challenge then becomes determining sufficient asymptotic relations for the sequences $(\epsilon_n)_{n\geq 1}$ and $(m_n)_{n\geq 1}$ to ensure desirable statistical properties of our estimator.

\subsubsection{Asymptotics for the pixel width}\label{sec:results_growingT1}

We relate the domain size with an appropriate pixel width by defining \emph{resolution} in the context of excursion sets of random fields, inspired by the notion of optical resolution.

\begin{defi}\label{def:resolution}
Define the random variable
\begin{equation*}
\Lambda_X^T(u) := \min\Big\{\mathrm{reach}\big( E_X(T,u)\big),\ \mathrm{reach}\big(T \setminus E_X(u)\big),\ \mathrm{reach}\big(\mathcal{Y}_{X}^T(u)\big)\Big\}.
\end{equation*}
For $\lambda\in\R^+$, we say that `` $E_X(u)$ \textit{is resolved by} $\lambda$ \textit{in} $T$ " whenever the random event $\{\lambda < \Lambda_X^T(u)\}$ occurs.
\end{defi}

This makes $\Lambda_X^T(u)$ a random geometrical description of $E_X(u)$ in the domain $T$: $\Lambda_X^T(u)$ is the supremum of the set of $\lambda\in\R^+$ such that one can roll a ball of radius $\lambda$ along both sides of the curve $E^\partial_X(T,u)$, and that the distances between points in $\mathcal{Y}_X^T(u)$ are all at least $2\lambda$. Figure~\ref{fig:reach} clarifies some of the notions introduced in Definition~\ref{def:resolution}. This definition allows us to relate the domain size with the pixel width, since the estimation error can be bounded in the case where $E_X(u)$ is resolved by $m_n\epsilon_n$ in $T_n$ (see the proof of Theorem~\ref{thm:consistent_Phat}).

\begin{figure}
    \centering
    \includegraphics[width=0.4\linewidth]{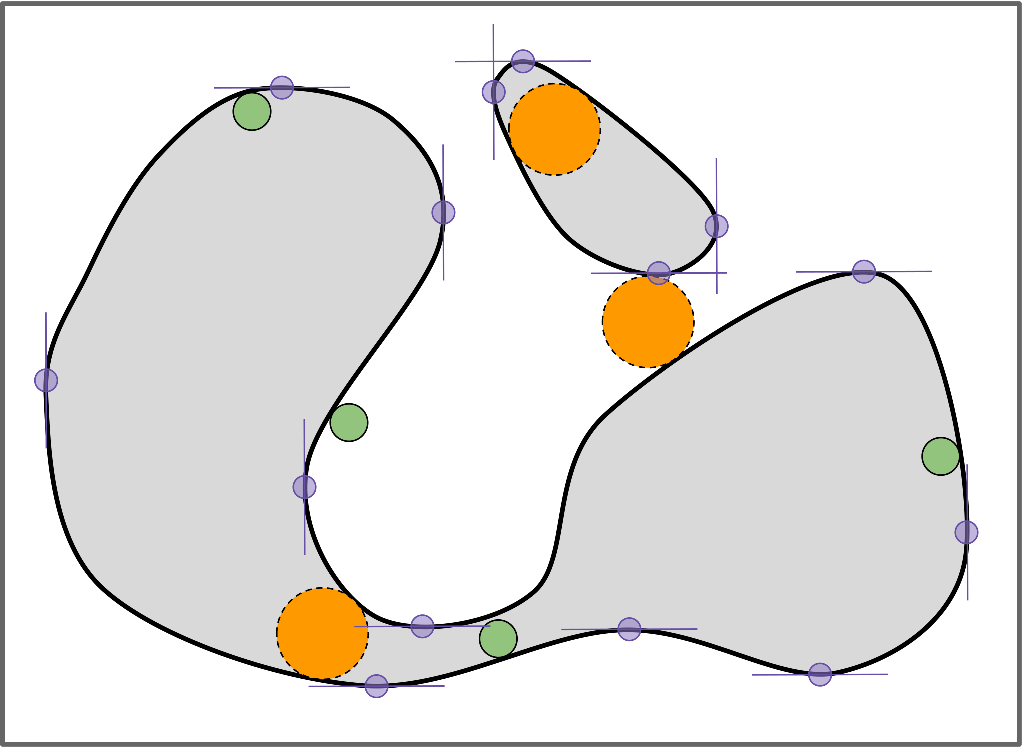}
    \put(-135,85){$E_X(u)$}
    \put(-153,105){$T$}
    \caption{Illustration of the notions of reach and resolution in Definition~\ref{def:resolution}. The reach of $E_X(T,u)$ is greater than the radius, $r_{\mathrm{green}}$, of the small green circles with solid border. The reach of $T \setminus E_X(u)$ is also greater than $r_{\mathrm{green}}$. Moreover, the minimum distance between points in $\mathcal{Y}_{X}^T(u)$, highlighted in purple, exceeds $2r_{\mathrm{green}}$. Therefore, $E_X(u)$ is \textit{resolved} by $r_{\mathrm{green}}$ in $T$ (see Definition~\ref{def:resolution}). Conversely, it is clear that $E_X(u)$ is not resolved in $T$ by the radius of the larger orange circles with dashed border.}
    \label{fig:reach}
\end{figure}
%

\begin{rema}\label{rem:positive_as}
Under Assumptions \ref{ass:basic} and \ref{ass:weak}, the random sets $E_X(T,u)$ and $T \setminus E_X(u)$ have positive reach almost surely, since $E_X(u)$ and $E_{-X}(u)$ have a twice differentiable boundary everywhere in $T$, almost surely, for all $u\in\R$. The intersection of these sets with the compact rectangle $T$ guarantees that the reach of each intersection is positive \citep[][p. 541]{bierme2019}. The minimum distance between points in $\mathcal{Y}_{X}^T(u)$ is positive by equation~\eqref{eqn:gradient_ass_k} and the compactness of $T$.
Therefore, $\Lambda_X^T(u)$ in Definition~\ref{def:resolution}  is almost surely positive for all $u\in\R$. Equivalently, for any $u\in\R$,
\begin{equation*}
\P\big(\liminf_{\lambda\rightarrow0}\big\{ \lambda < \Lambda_X^T(u)\big\}\big) = 1,
\end{equation*}
\textit{i.e.}, with probability 1, there exists a sufficiently small positive $\lambda$ that resolves $E_X(u)$ in $T$.
\end{rema}


With the notion of resolution established, we state an important convergence result for the sequence of growing domains $(T_n)_{n\geq 1}$ under general regularity assumptions.

\begin{prop}\label{prp:Tn} Let $X$ be a random field satisfying Assumptions~\ref{ass:basic},~\ref{ass:weak}, and~\ref{ass:integrable_quantities}. Let $(m_n)_{n\geq 1}$ be a non-decreasing sequence in $\N^+$ such that $m_n/n\rightarrow\infty$.
Let $(\epsilon_n)_{n\geq 1}$ be a non-increasing sequence in $\R^+$ satisfying $\epsilon_n = O\big(m_n^{-3/2}\big)$. Moreover, suppose that $2t$ is an integer multiple of $\epsilon_n$ for all $n\in\N^+$, and $\P\big(m_n\epsilon_n < \Lambda_X^{T_n}(u)\big)\rightarrow 1$ as $n\rightarrow\infty$. Then for any $u\in\R$,
\begin{equation*}
    \frac{\hat{P}^{(2)}_{X}(\epsilon_{n}, m_{n}; T_n, u) - {P}_{X}^{T_n}(u)}{\sqrt{\nu(T_n)}} \stackrel{\P}{\longrightarrow} 0,
\end{equation*}
as $n\rightarrow\infty$.
\end{prop}
The proof of Proposition~\ref{prp:Tn} is postponed to Section~\ref{sec:proofs}.

\begin{rema}
One example of a sequence $(\epsilon_n)_{n\geq 1}$ satisfying the constraints in Proposition~\ref{prp:Tn} is constructed by letting $\epsilon_n$ be the largest element in the sequence $(2t/k)_{k\geq 1}$ such that $\epsilon_n \leq m_n^{-3/2}$ and $\P\big(\Lambda_X^{T_n}(u) \leq m_n\epsilon_n\big) \leq 1/n$, where $\Lambda_X^{T_n}(u)$ is defined in Definition~\ref{def:resolution}. Such a sequence $(\epsilon_n)_{n\geq 1}$ exists since $\P(\Lambda_X^{T_n}(u)\leq 0)=0$ for all $n\in\N^+$ as discussed in Remark \ref{rem:positive_as}. The idea is to have the sequence $\lambda_n := m_n\epsilon_n$ tend to 0 faster than the quantiles of $\Lambda_X^{T_n}(u)$, which is difficult to verify analytically. However, in practice, for a given realization of $E_X(u)$, one can estimate $\Lambda_X^{T}(u)$ by first estimating the reach of the sets $E_X(T,u)$ and $T\setminus E_X(u)$ \citep{aamari2019, cotsakis2022_2} and the vector coordinates of the points in $\mathcal{Y}_X^T(u)$, defined in~\eqref{eqn:curly_Y}.
\end{rema}

Proposition~\ref{prp:Tn} establishes that for a large class of random fields, as the domain grows and the grid spacing decreases, the error in the perimeter estimation is negligible compared to the side length of the domain. Such a comparison is made possible by the conditions on the sequences $(m_n)_{n\geq 1}$ and $(\epsilon_n)_{n\geq 1}$, since the indexing variable $n$ is proportional to the side length of $T_n$.

\subsubsection{Asymptotic normality of the perimeter estimator}\label{sec:asymptotic_normality}

In this section, we prove a multivariate Central Limit Theorem for our estimator as stated in Theorem~\ref{thm:CLT} below, based on the results from \cite{iribarren1989}. The interested reader is also referred to \cite{cabana1987}.  \smallskip

First, we recall two important notions regarding the random fields for which the theorem applies. Recall that a random field $X = \{X(s):s\in\R^2\}$ is said to be \textit{affine} if it is equal in distribution to $\{Y(As):s\in\R^2\}$, where $Y$ is stationary, isotropic, and $A$ is a positive-definite $2\times2$ matrix. Consequentially, the resulting $X$ is stationary but may be anisotropic. Note that it is common in geostatistics literature to use the nomenclature \textit{geometric anisotropy} when referring to affine random fields \citep{chiles2009}.  \smallskip

In the case of $X$ affine, a useful expression for $\E[P_X^{T}(u)]$, when it exists, is provided in \citet[Section~1.1]{cabana1987}; that is,
\begin{equation}\label{eqn:ellipse}
\E[P_X^{T}(u)] = \frac{\mathrm{ellipse(\lambda_1,\lambda_2)}}{2\pi}\E[P_Y^{T}(u)],
\end{equation}
with $\lambda_1$ and $\lambda_1$ denoting the eigenvalues of $A$, and $\mathrm{ellipse(a,b)}$ denoting the perimeter of an ellipse with semi-minor and semi-major axes $a$ and $b$.   \smallskip

Recall that $X$ is said to be \textit{strongly mixing}, or \textit{uniformly mixing}, if there exists a function $\psi(\rho):\R^+\rightarrow\R^+$ tending to 0 as $\rho\rightarrow\infty$, such that for any two measurable sets $S_1,S_2\subset\R^2$ that satisfy $\inf\{||s_1-s_2||_2:s_1\in S_1, s_2 \in S_2\} =: \rho > 0$, and for any events $A_1$ and $A_2$ in the the sigma fields generated by $\{X(s):s\in S_1\}$ and $\{X(s):s\in S_2\}$ respectively, it holds that $|\P(A_1\cap A_2)-\P(A_1)\P(A_2)| < \psi(\rho)$.   \smallskip

Under the assumption that the underlying random field is affine and strongly mixing, we prove the  multivariate central limit theorem for our estimator. The proof of Theorem~\ref{thm:CLT} is postponed to Section~\ref{sec:proofs}.



\begin{theo}\label{thm:CLT}
Let $X$ be a stationary, affine, strongly mixing random field satisfying Assumptions~\ref{ass:basic}--\ref{ass:integrable_quantities}. With $\nabla X$ denoting the gradient of $X$, suppose that the joint density function of $(X,\nabla X)$ is bounded. Let $k\in\N^+$ and fix the vector $\mathbf{u}:=(u_1,\ldots,u_k)\in\R^k$ such that $u_i\neq u_j$ for $1\leq i<j\leq k$. Let the sequences $(m_n)_{n\geq 1}$ and $(\epsilon_n)_{n\geq 1}$ satisfy the constraints in Proposition~\ref{prp:Tn} for all $u_j$, with $j=1,\ldots,k$. Let
$$\hat{P}_X^{(2)}(\epsilon_n, m_n; T_n,\mathbf{u}):= \big(\hat{P}_X^{(2)}(\epsilon_n, m_n; T_n,u_1),\ldots,\hat{P}_X^{(2)}(\epsilon_n, m_n; T_n,u_k)\big)$$
and
$$P_X^{T_n}(\mathbf{u}):= \big(P_X^{T_n}(u_1),\ldots,P_X^{T_n}(u_k)\big).$$
Then there exists a finite, non-degenerate (\textit{i.e.}, full-rank) covariance matrix $\Sigma(\mathbf{u})$ such that
\begin{equation}\label{eqn:CLT_statement}
    \frac{\hat{P}_X^{(2)}(\epsilon_n, m_n; T_n,\mathbf{u}) - \E[P_X^{T_n}(\mathbf{u})]}{\sqrt{\nu(T_n)}} \stackrel{\d}{\longrightarrow} \mathcal{N}_k\big(\mathbf{0},\Sigma(\mathbf{u})\big),\qquad n\rightarrow\infty,
\end{equation}
with $\E[P_X^{T_n}(u_j)]$ as in \eqref{eqn:ellipse} for all $u_j$, $j=1,\ldots,k$. The elements of $\Sigma(\mathbf{u})$ are of the form
\begin{equation}\label{eqn:var_Sigma}
\Sigma_{ij}(\mathbf{u}) = \int_{\R^2} H_s(u_i,u_j)\ \d s,
\end{equation}
where
\begin{align*}
H_s(u_i,u_j) =& g_s(u_i,u_j)
\E\Big[||\nabla X(0)||_2 ||\nabla X(s)||_2\ \big|\ X(0)= u_i, X(s) = u_j\Big]\\
&- f(u_i)f(u_j)
\E\Big[||\nabla X(0)||_2\ \big|\ X(0) = u_i\Big]\E\Big[||\nabla X(s)||_2\ \big|\ X(s) = u_j\Big],
\end{align*}
with
$f$ denoting the marginal density function of $X$, and $g_s$, the joint density function of $\big(X(0),X(s)\big)$.
\end{theo}

As seen in the proof of Theorem~\ref{thm:CLT}, the rescaled limiting Gaussian distribution of our perimeter estimator---in our pixelated framework---coincides with that of $P_X^{T_n}(u)$, the true perimeter in the continuous framework.\smallskip


%

Corollary~\ref{cor:gaussian}, stated below, provides a succinct set of conditions on $X$ that imply the result of Theorem~\ref{thm:CLT}. In particular, the additional assumption of Gaussianity of the underlying random fields is introduced.

\begin{coro}\label{cor:gaussian}
Suppose that there exists a positive-definite matrix $A$ such that the random field $X$ is equal in distribution to $\{Y(As):s\in\R^2\}$, for some $C^2$, stationary, isotropic, centered, Gaussian random field $Y$ with covariance function $r(h)$, $h\in\R^2$. Define
$$\Psi(s) = \max\Big\{ |r(s)|,\ |r_1(s)|,\ |r_2(s)|,\ |r_{11}(s)|,\ |r_{22}(s)|,\ |r_{12}(s)| \Big\},$$
for $s\in\R^2$, where $r_i := \partial r / \partial s_i$ and $r_{ij} := \partial^2 r / (\partial s_i \partial s_j)$ for $i,j\in\{1,2\}$. Suppose further that $\Psi(s)\rightarrow 0$ as $||s||_2\rightarrow\infty$, $\int_{\R^2}|\Psi(s)|\ \d s < \infty$, and $\int_{\R^2}r(s)\ \d s > 0$. Then the result of Theorem~\ref{thm:CLT} holds.
\end{coro}
The proof can be found  in  Section~\ref{sec:proofs}. We remark that a vast literature exists on the asymptotic distribution of level functionals of Gaussian random fields \citep{wschebor1985, meschenmoser2013, shashkin2013,  dibernardino2017, beliaev2020, dibernardino2020}, in which case, the asymptotic variance-covariance matrix in \eqref{eqn:var_Sigma} can be written by projecting the Gaussian functionals of interest
onto the It\^{o}-Wiener chaos \citep[the interested reader is referred, for instance, to][]{kratz2001, EL16, mueller2017,kratz2018,berzin2021}.

\section{Simulation studies}\label{sec:sims}
In this section, we illustrate finite sample performances of our estimator $\hat P_X^{(2)}(\epsilon, m; T, \textbf{u})$ on simulated data. More precisely, we wish to showcase the results of Proposition~\ref{prp:L1} and Theorem~\ref{thm:CLT}. Furthermore, we aim to compare the estimators constructed from the norms $p=1$ and $p=2$ in~\eqref{eqn:phat}. Our simulation studies are implemented both for anisotropic (see Section~\ref{sec:aniso}) and isotropic (see Section~\ref{sec:iso}) random fields. In addition, we provide an adaptive method for choosing the hyperparameter $m$ for the estimator $\hat P_X^{(2)}(\epsilon, m; T, u)$ (see Section~\ref{sec:hyperparam}). The random fields used in each simulation are elements of the class in Example~\ref{exa:general_X} below.

\begin{exam}\label{exa:general_X}
Let $Y$ be a stationary, isotropic, centered, Gaussian random field with a Mat\'ern covariance function
$$r(h):=\frac{2^{1-\nu}}{\Gamma(\nu)}\big(\sqrt{2\nu}||h||_2\big)^{\nu}K_\nu(\sqrt{2\nu}||h||_2),\qquad h\in\R^2,$$
where $K_\nu$ is the modified Bessel function of the second kind and $\nu=2.5$. To clarify, the range parameter in the covariance function is fixed as 1.

Let $\{X(s;\sigma_1,\sigma_2,\theta):s\in\R^2\}$ be a random field equal in distribution to $\{Y(A s):s\in\R^2\}$, where
\begin{equation}\label{eqn:A_theta}
A := \begin{bmatrix}
\sigma_1 & 0\\
0\ & \sigma_2
\end{bmatrix}
\begin{bmatrix}
\cos\theta & \sin\theta \\
-\sin\theta\ & \cos\theta
\end{bmatrix},
\end{equation}
$\sigma_1, \sigma_2 \in \R^+$, $\sigma_1 \geq \sigma_2$, and $\theta\in[0,\pi)$.
In this way, $X(\cdot;\sigma_1,\sigma_2,\theta)$ is affine with affinity parameters $k = (1-\sigma_2^2/\sigma_1^2)^{1/2}$ and $\theta$ \citep{cabana1987}. Notice that $X(\cdot;\sigma_1,\sigma_2,\theta)$ is also Gaussian with covariance function given by $r_X(h)= r(A h)$. Although $A$ is not necessarily positive-definite, there exists a unique positive-definite matrix $B$ with eigenvalues $\sigma_1$ and $\sigma_2$ such that $||A h||_2 = ||B h||_2$ for all $h\in\R^2$. Note also that $\sigma_1=\sigma_2$ if and only if $X$ is isotropic, in which case, $X$ does not depend on $\theta$.
\end{exam}
Throughout Section~\ref{sec:sims}, $X(\cdot;\sigma_1,\sigma_2,\theta)$ and $Y$ denote the random fields in Example~\ref{exa:general_X}. The former is sometimes abbreviated as $X$, and the dependence on $\sigma_1$, $\sigma_2$, and $\theta$ should be understood implicitly. The results in this section can be reproduced using the code made available at \url{https://github.com/RyanCotsakis/excursion-sets}.

\subsection{A proxy for the true perimeter}\label{sec:proxy}

In what follows, the R package \texttt{RandomFields} is used to generate realizations of random fields on regular grids. However, when simulating the random field $X(\cdot;\sigma_1,\sigma_2,\theta)$ in this way, it is impossible to infer the exact value of $P_X^T(u)$ for any level $u\in\R$ due to the discretization of the domain $T$.
To overcome this issue, a proxy is used for the true perimeter.
In Appendix B of \cite{bierme2021}, the authors introduce an estimator that they show to be multigrid convergent for ${P}_{X}^T(u)$, for any $u\in\R$.
Moreover, the estimator takes as its arguments the values of $X$, a random field with $C^2$ sample paths, evaluated on a regular grid, \emph{i.e.}, $X(s_{i,j})$ for $i,j\in I^{(T,\epsilon)}$---precisely the output of the simulation from the \texttt{RandomFields} package. For a pixel width of $\epsilon$, denote this estimator by $\tilde{P}_{X}(\epsilon;T,u)$. Notice that $\tilde{P}_{X}(\epsilon; T,u)$ requires more information than $\hat P_X^{(2)}(\epsilon, m; T, {u})$.
While $\tilde{P}_{X}$ has access to the value of $X$ evaluated on the regular square tiling $\mathcal{G}^{(T,\epsilon)}$, defined in~\eqref{eqn:grid_def}, $\hat P_X^{(2)}$ only has access to the binary black-and-white matrix $\zeta_{X}^{(T,\epsilon)}(u)$, defined in~\eqref{eqn:zeta_def}.
\smallskip

Convergence of $\tilde{P}_{X}(\epsilon_n;T,u)$ to ${P}_{X}^T(u)$ in $L^1(\Omega)$ follows from the same arguments that we use in the proof of our Proposition~\ref{prp:L1}.
Therefore, for any sequence $(h_n)_{n\geq 1}$,
\begin{equation}\label{eqn:L1_equiv}
\big|h_n - \tilde{P}_{X}(\epsilon_n;T,u)\big| \stackrel{L^1}{\longrightarrow} 0
\iff
h_n \stackrel{L^1}{\longrightarrow} {P}_{X}^T(u)
\end{equation}
as $n\rightarrow\infty$.

\subsection{The anisotropic case}\label{sec:aniso}

None of the assumptions established thus far prohibit anisotropy. In fact, all of the results developed in Section~\ref{sec:main_results} are applicable to all of the random fields parameterized as in Example~\ref{exa:general_X}. In Sections~\ref{sec:aniso_angle},~\ref{sec:aniso_convergence}, and~\ref{sec:aniso_clt}, we consider such random fields that are anisotropic (\textit{i.e.}, parametrized by $\sigma_1\neq\sigma_2$). To avoid confusion, we consistently choose $(\sigma_1,\sigma_2) = (2,0.5)$.


\subsubsection{Mean perimeter estimate as a function of the angle $\theta$}\label{sec:aniso_angle}

The random fields
$X$
in Example~\ref{exa:general_X} parametrized by $(\sigma_1,\sigma_2) = (2,0.5)$ and several $\theta\in[0,\pi / 2]$ are simulated in the domain $T=[-2.5,2.5]^2$, discretized into $256\times 256$ pixels. With $\epsilon$ denoting the resulting pixel width, the performances of the estimators $(\pi / 4) \hat P_{X}^{(1)}(\epsilon;T,u)$ and $\hat P_{X}^{(2)}(\epsilon, m;T,u)$ with $m=11$ are compared at the level $u=0.5$.
For each of the several values of $\theta$ chosen in $[0,\pi / 2]$, 200 independent replications of $X(\cdot;2,0.5,\theta)$ are simulated in the domain $T$ and the mean error in the estimates of $P_{X}^T(0.5)$ is plotted for each of the two estimators: the sample means of $(\pi / 4)\hat P_{X}^{(1)}(\epsilon;T,0.5) - \tilde{P}_{X}(\epsilon;T,0.5)$ (shown in green) and $\hat P_{X}^{(2)}(\epsilon, 11;T,0.5)$ (shown in blue), and $\tilde{P}_{X}(\epsilon;T,0.5) - \tilde{P}_{X}(\epsilon;T,0.5)$ (shown in black) in Figure~\ref{fig:aniso} (c). Notice that
$\E\big[(\pi / 4)\hat P_{X}^{(1)}(\epsilon;T,0.5)\big]$ depends on $\theta$, since $\E\big[P_{X}^T(0.5)\big] = 19.4$ for all $\theta$. The latter expectation is computed via equation~\eqref{eqn:ellipse} and the Gaussian Kinematic Formula in \citet[Theorem~15.9.5]{adler2007}.
The sample average of $\hat P_{X}^{(2)}(\epsilon, 11;T,0.5) - \tilde{P}_{X}(\epsilon;T,0.5)$ shown in Figure~\ref{fig:aniso} is nearly 0 for all $\theta$, thus supporting our claim that that our estimator adapts to anisotropic random fields.

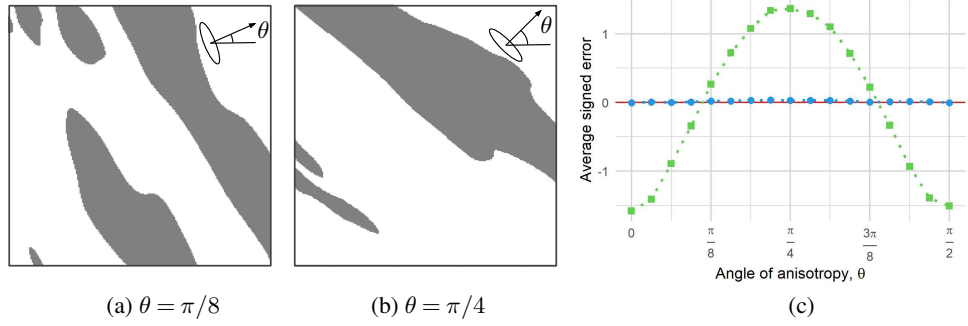
\begin{figure}
\centering
\includegraphics[width=0.9\linewidth]{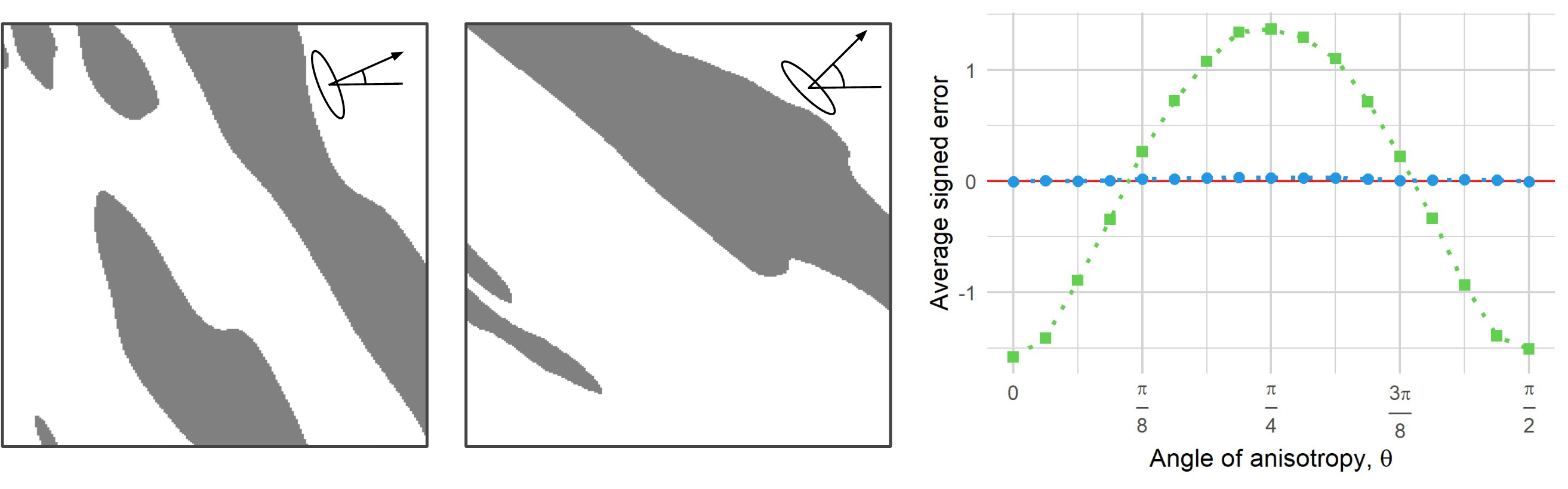}
\put(-328,-8){(a) $\theta = \pi / 8$}
\put(-228,-8){(b) $\theta = \pi / 4$}
\put(-70,-8){(c)}
\put(-271,96){$\theta$}
\put(-164,98){$\theta$}
\caption{Illustration of the effect of anisotropy on the perimeter length estimation. The anisotropic random field ${X(\cdot;2,0.5,\theta)}$ is described in Example~\ref{exa:general_X}. Here, $T=[-2.5,2.5]^2$ and $\epsilon=5/255$. Panels (a, b): a realization of $E_{X}(T,0.5)$ shown as the dark region for the corresponding value of $\theta$. The matrix $A$, defined in~\eqref{eqn:A_theta}, maps the drawn ellipse to a circle. Panel (c): for several $\theta\in[0,\pi / 2]$, 200 independent realizations of ${X}$ are simulated, and the mean values of $(\pi / 4)\hat P_{X}^{(1)}(\epsilon;T,0.5)- \tilde{P}_{X}(\epsilon;T,0.5)$ (green squares) and $\hat P_{X}^{(2)}(\epsilon, 11;T,0.5)- \tilde{P}_{X}(\epsilon;T,0.5)$ (blue circles) are plotted.}
\label{fig:aniso}
\end{figure}
%

\subsubsection{Convergence in mean in the anisotropic case}\label{sec:aniso_convergence}

Let $\lfloor \cdot \rfloor$ denote the floor function. For $n\in\N^+$, fix the domain $T = [-2.5,2.5]^2$ and let
\begin{equation}\label{eqn:sequences_eps_m}
M_n = \lfloor 10 n^{3/2} \rfloor,\qquad m_n=n,\qquad \mathrm{and}\qquad\epsilon_n = 5/(M_n-1),
\end{equation}
so that the constraints in Theorem~\ref{thm:consistent_Phat} and Proposition~\ref{prp:L1} are satisfied. Let $X(\cdot;2,0.5,0)$ be the random field in Example~\ref{exa:general_X} associated to $(\sigma_1,\sigma_2,\theta)=(2,0.5,0)$. As noted in Remark \ref{rem:pixels}, the quantity $M_n$ should be interpreted as the pixel density of the discretized domain $T$, and $\epsilon_n$ should be understood as the corresponding pixel width. Figure~\ref{fig:m23} provides two illustrations of $E_X(u)$, with $u=0.5$, in the domain $T$; one containing $M_2\times M_2$ pixels, and another containing of $M_3\times M_3$ pixels. In this   study, $\E[P_{X}^T(0.5)] = 21.3$ (computed via equation~\eqref{eqn:ellipse} and the Gaussian Kinematic Formula in \citet[Theorem~15.9.5]{adler2007}). \smallskip

To illustrate the convergence of $\hat P_X^{(2)}(\epsilon_n, m_n; T, 0.5)$ to ${P}_{X}^T(0.5)$ in $L^1(\Omega)$, the left-hand side of equation~\eqref{eqn:L1_equiv} is shown numerically with $h_n = \hat P_X^{(2)}(\epsilon_n, m_n; T, 0.5)$.
Figure~\ref{fig:mae_aniso} shows how the mean absolute error (MAE) of the approximation of $\tilde{P}_{X}(\epsilon_n;T,0.5)$ (the proxy for $P_X^T(0.5)$; see Section~\ref{sec:proxy}) by the estimator $\hat P_X^{(2)}(\epsilon_n, m_n; T, 0.5)$ (shown in blue) approaches 0 as $n\rightarrow\infty$. There is no convergence result for the estimator  $(\pi / 4)\hat P_X^{(1)}(\epsilon_n; T, 0.5)$ (shown in green) since it is not well-suited for anisotropic random fields.

\begin{figure}
    \centering
    \includegraphics[width=\linewidth]{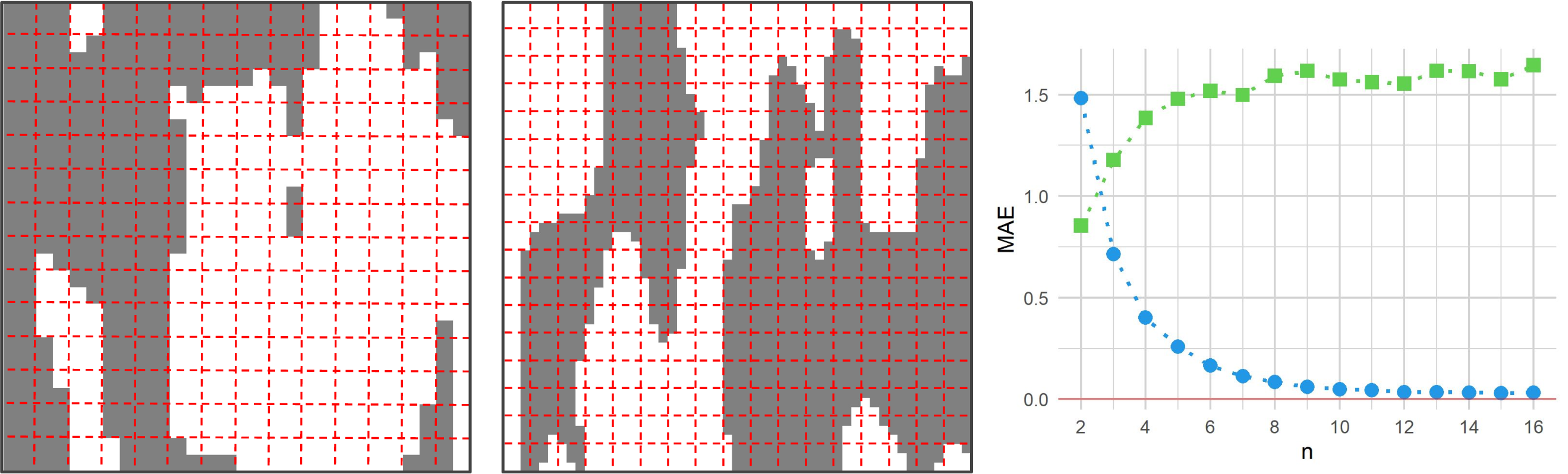}
    \put(-355,-12){(a) $n = 2$}
    \put(-230,-12){(b) $n = 3$}
    \put(-72,-12){(c)}
    \caption{The case of decreasing pixel width with the domain $T=[-2.5,2.5]^2$ fixed. Here, $u=0.5$; $M_n$, $m_n$, and $\epsilon_n$ are given in~\eqref{eqn:sequences_eps_m}; and $X$ in Example~\ref{exa:general_X}, parametrized by $(\sigma_1,\sigma_2,\theta) = (2,0.5,0)$, is anisotropic. Panel (a): the excursion set (shown as the dark region) is generated using $M_2\times M_2$ pixels, and the dashed red lines have a spacing of $2\epsilon_2$, where $\epsilon_2$ is the pixel width. Panel (b): the size of the image (measured in pixels) is $M_3\times M_3$, and the dashed red lines have a spacing of $3\epsilon_3$, where $\epsilon_3$ is the pixel width. Panel (c): the approximation of $\tilde{P}_{X}(\epsilon_n;T,0.5)$ by $(\pi / 4)\hat P_X^{(1)}(\epsilon_n; T, 0.5)$ (green squares) and by $\hat P_X^{(2)}(\epsilon_n, m_n; T,0.5)$ (blue circles) is shown for different values of $n$. For each $n$, the MAE of the approximations are calculated from 500 independent replications of the process $X$. }
    \label{fig:mae_aniso}
\end{figure}

\subsubsection{Asymptotic normality in the anisotropic case}\label{sec:aniso_clt}

To illustrate the  Central Limit Theorem  for multiple levels (see Theorem~\ref{thm:CLT}), we compute $\hat P_X^{(2)}(\epsilon, m; T, \textbf{u})$ in a large domain $T = [-15,15]^2$ divided into $1024\times 1024$ pixels, with $m=7$, $\textbf{u} = (0,0.5,1)$, and $X$ as in Example~\ref{exa:general_X} with $(\sigma_1,\sigma_2) = (2,0.5)$ and $\theta=\pi / 4$. Figure~\ref{fig:qqplots_aniso} shows how the distribution of the random vector $\hat P_X^{(2)}(\epsilon, m; T, \textbf{u})$
is close to a 3-variate normal distribution with mean $\E[P_X^{T}(\textbf{u})] = (793, 700, 481)$ (computed via equation~\eqref{eqn:ellipse}).\smallskip

\begin{figure}
    \centering
    \includegraphics[width=0.8\linewidth]{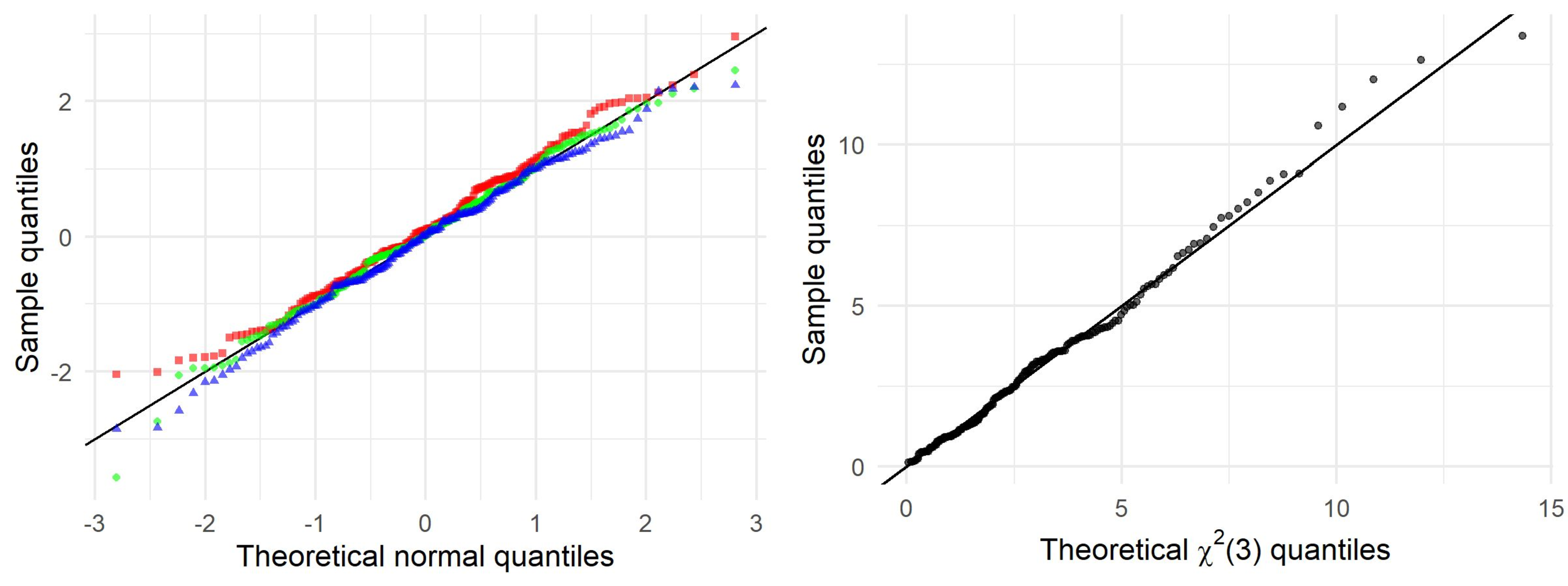}
    \put(-242,-9){(a)}
    \put(-79,-9){(b)}
    \caption{
An illustration of the asymptotic normality of our estimator for the anisotropic random field $X(\cdot;2,0.5,\pi / 4)$ in Example~\ref{exa:general_X}. We simulated 200 independent replications of the vector $\hat P_{X}^{(2)}(\epsilon, m; T, \textbf{u})$ with $\textbf{u} = (0,0.5,1)$, $T=[-15,15]^2$, $m=7$, $\epsilon = 30/1023$.
    Panel (a): the margins of $\hat P_X^{(2)}(\epsilon, m; T, \textbf{u}) - \E[P_X^{T}(\textbf{u})]$, rescaled using the sample variances, plotted on a normal qq-plot. Panel  (b): the squared Mahalanobis distance of $\hat P_X^{(2)}(\epsilon, m; T, \textbf{u})$ to $\E[P_X^{T}(\textbf{u})]$, calculated via the sample covariance matrix of $\hat P_X^{(2)}(\epsilon, m; T, \textbf{u})$, plotted against the quantiles of a $\chi^2(3)$ random variable with 3 degrees of freedom.}
    \label{fig:qqplots_aniso}
\end{figure}

For each component $u_i$ of $\textbf{u}$, we test the null hypothesis that $\hat P_X^{(2)}(\epsilon, m; T, u_i)$ follows a Gaussian distribution using the Shapiro-Wilk test. The resulting $p$-values from the tests are 0.39, 0.49, and 0.31, respectively. Thus, the hypothesis of Gaussianity cannot be rejected at a significant level for any margin of $\hat P_X^{(2)}(\epsilon, m; T, \textbf{u})$.
Using the R package \texttt{mvnormtest} \citep{mvnormtest}, we test the null hypothesis that $\hat P_X^{(2)}(\epsilon, m; T, \textbf{u})$ follows a multivariate normal distribution with a multivariate Shapiro-Wilk test. The test statistic corresponds to a $p$-value of 0.14, hence, multivariate normality cannot be rejected at a significant level.

\subsection{The isotropic case}\label{sec:iso}
In what follows, $Y$ denotes the isotropic random field in Example~\ref{exa:general_X}. This isotropic case allows for a fair comparison between the estimators $(\pi / 4)\hat P_{Y}^{(1)}(\epsilon;T, u)$ and $\hat P_Y^{(2)}(\epsilon, m; T, u)$.

\subsubsection{Convergence in mean in the isotropic case}\label{sec:iso_convergence}

The experiment in Section~\ref{sec:aniso_convergence} is repeated for the isotropic random field $Y$. Figure~\ref{fig:m23} summarizes the new results. The MAE of the approximation of $\tilde{P}_{Y}(\epsilon_n;T,0.5)$ by $\hat P_Y^{(1)}(\epsilon_n; T, 0.5)$ (shown in green) tends to a positive value, so by~\eqref{eqn:L1_equiv}, $(\pi / 4)\hat P_Y^{(1)}(\epsilon_n; T, 0.5)$ does not converge to ${P}_{Y}^T(0.5)$ in $L^1(\Omega)$, even though $\E\big[(\pi / 4)\hat P_Y^{(1)}(\epsilon_n; T, 0.5)\big]\rightarrow \E\big[{P}_{Y}^T(0.5)\big]$ as $n\rightarrow\infty$ (see equation~\eqref{eqn:bierme_converge}). The interested reader is referred to Theorem 3 in \cite{bierme2021}. For reference, $\E[P_Y^T(0.5)]= 15.6$ (computed via the Gaussian Kinematic Formula in \citet[Theorem~15.9.5]{adler2007}).

\begin{figure}
    \centering
    \includegraphics[width=\linewidth]{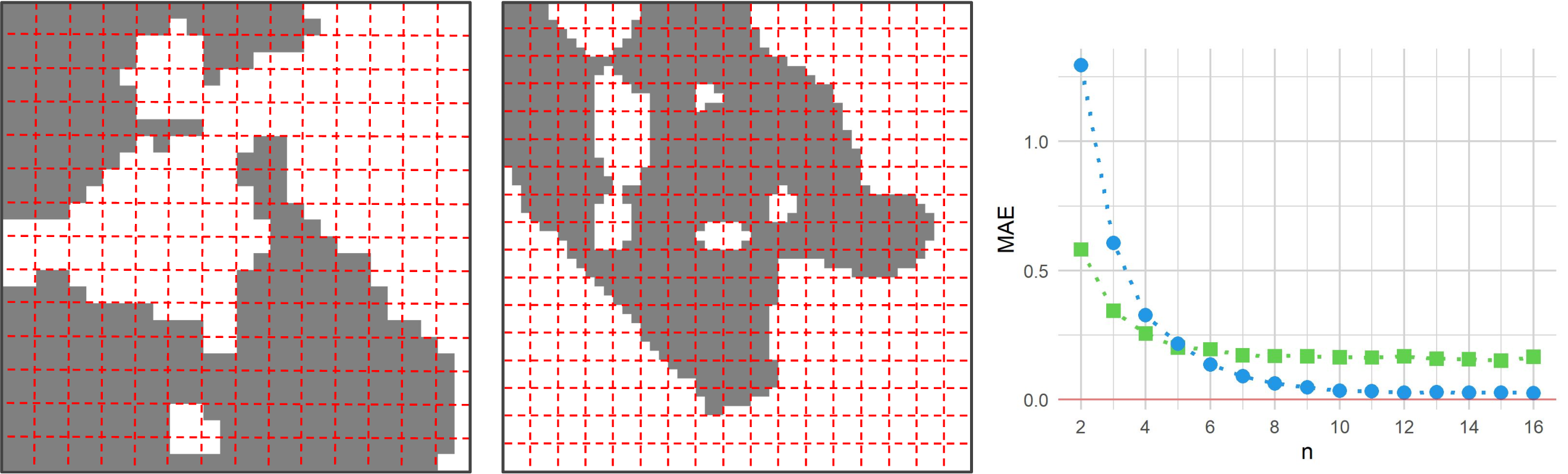}
    \put(-355,-12){(a) $n = 2$}
    \put(-230,-12){(b) $n = 3$}
    \put(-72,-12){(c)}
    \caption{The case of decreasing pixel width and fixed domain $T=[-2.5,2.5]^2$, where $u=0.5$; $M_n$, $m_n$, and $\epsilon_n$ are given in~\eqref{eqn:sequences_eps_m}; and $Y$ is the isotropic random field in Example~\ref{exa:general_X}. See the caption of Figure~\ref{fig:mae_aniso} for a more detailed description of each panel.}
    \label{fig:m23}
\end{figure}

\subsubsection{Asymptotic normality in the isotropic case}\label{sec:iso_clt}

We repeat the experiment in Section~\ref{sec:aniso_clt}, which tests the asymptotic normality of our estimator, but now with $Y$ as the underlying random field. The $p$-values corresponding to the Gaussianity tests for the levels $u=0$, 0.5, and 1 are 0.80, 0.68, and 0.43, respectively. For the multivariate normality test, the resulting $p$-value is 0.37. The same diagnostic plots in Section \ref{sec:aniso_clt} are provided in Figure~\ref{fig:qqplots} for this isotropic case.

\begin{figure}
    \centering
    \includegraphics[width=0.8\linewidth]{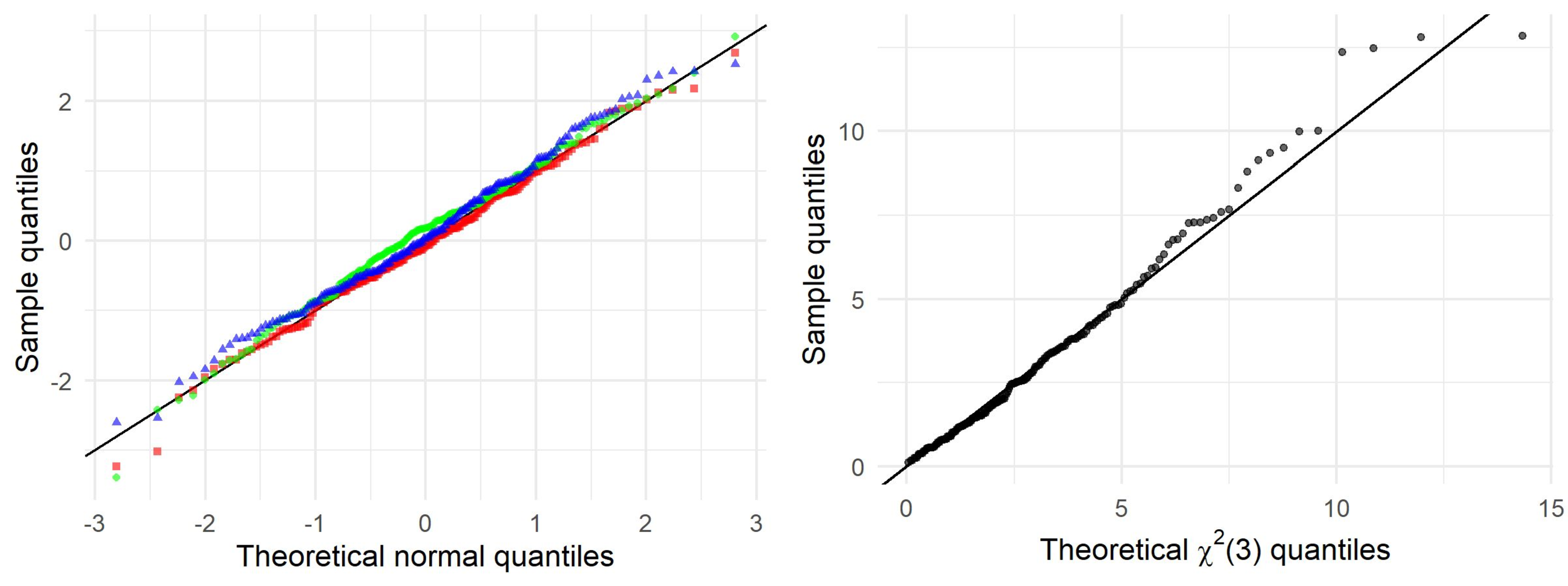}
    \put(-242,-9){(a)}
    \put(-79,-9){(b)}
    \caption{An illustration of the asymptotic normality of our estimator when considering the isotropic random field $Y$ in Example~\ref{exa:general_X}. We simulated 200 independent replications of the vector $\hat P_Y^{(2)}(\epsilon, m; T, \textbf{u})$ with $\textbf{u} = (0,0.5,1)$, $T=[-15,15]^2$, $m=7$, $\epsilon = 30/1023$. See the caption of Figure~\ref{fig:qqplots_aniso} for a description of each panel.}
    \label{fig:qqplots}
\end{figure}

\subsection{Hyperparameter selection}\label{sec:hyperparam}
In practice, sampling locations often have a fixed spacing, and it is not possible to further decrease the grid spacing in the discretization. In these cases, the pixel width $\epsilon$ is a feature of the data. So, to use $\hat P_X^{(2)}(\epsilon, m; T,u)$ (for an arbitrary model $X$), the hyperparameter $m$ must be chosen appropriately. As a rule-of-thumb, empirical studies suggest that it is reasonable to choose
\begin{equation}\label{eqn:mXT}
m = m_X^T := \big\lfloor C\epsilon^{-2/3}\big\rfloor,
\end{equation}
with
\begin{equation*}
C := \frac 13\bigg(\frac{\nu(T)}{N_{cc} + N_{holes}}\bigg)^{1/3},
\end{equation*}
where $N_{cc}$ (\emph{resp.} $N_{holes}$) corresponds to the number of connected components (\emph{resp.} holes) of $E_X(T,u)$. For a sequence $(\epsilon_n)_{n\geq 1}$ tending to 0, the corresponding sequence $(m_n)_{n\geq 1}$ determined by~\eqref{eqn:mXT} satisfies the asymptotic relationship required by Theorem~\ref{thm:consistent_Phat}.\smallskip

In practice, the quantities $N_{cc}$ and $N_{holes}$ can be estimated by considering the sites in $\mathcal{G}^{(T,\epsilon)}$ to be either 4-connected or 8-connected, and colouring each site based on its corresponding value in~$\zeta_X^{(T,\epsilon)}(u)$.\smallskip

Figures~\ref{fig:m_005} and~\ref{fig:m_02_60} showcase the performance of $\hat P_Y^{(2)}(\epsilon, m_Y^T; T,0)$, with $m_Y^T$ as in~\eqref{eqn:mXT}, for two different levels of discretization of the isotropic random field $Y$ in Example~\ref{exa:general_X}.

\begin{figure}
    \centering
    \includegraphics[width=\linewidth]{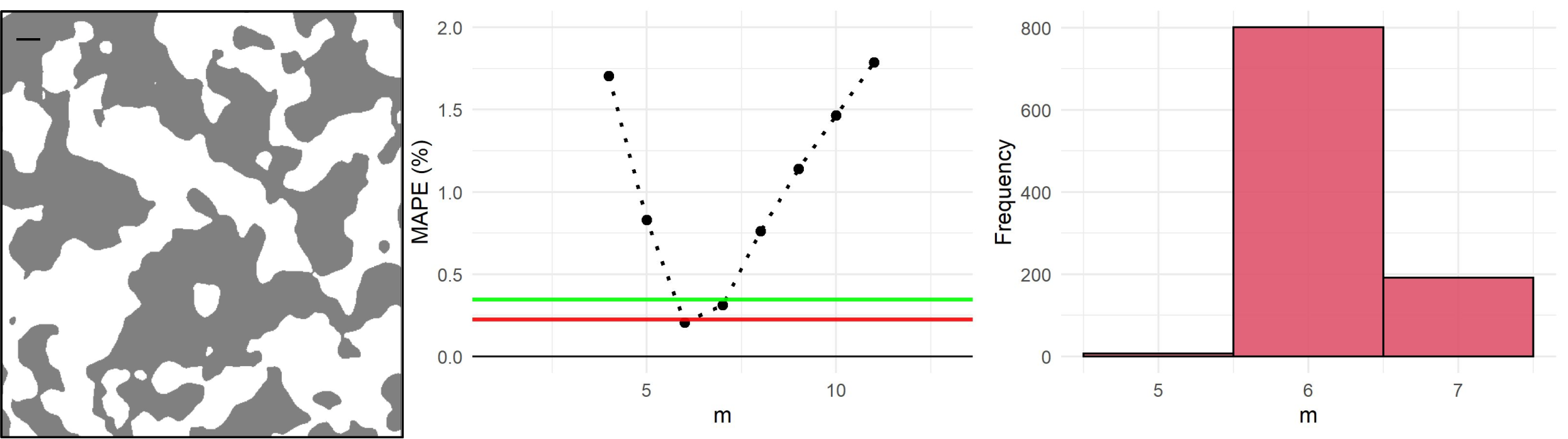}
    \put(-360,-11){(a)}
    \put(-224,-11){(b)}
    \put(-73,-11){(c)}
    \caption{Illustration of the influence of the hyperparameter $m$. The mean absolute percentage error (MAPE) of several perimeter estimators is calculated for 1000 independent replications of the stationary, isotropic, Gaussian random field $Y$ in Example~\ref{exa:general_X}, with $T=[-10,10]^2$, $u=0$, and $\epsilon = 20/511$. The proxy $\tilde{P}_{Y}(\epsilon;T,0)$ is used to represent the true perimeter $P_Y^T(0)$ for each sample path (see Section~\ref{sec:proxy}). Panel  (a): one particular realization of $E_Y(0)$ is depicted in $T$. Shown for scale in the top-left of the image is a line segment with length $30\epsilon$. Panel  (b): the points plotted in black correspond to the MAPE of $\hat P_Y^{(2)}(\epsilon, m; T,0)$ for various values of $m$. The green horizontal line (0.35\%) corresponds to the MAPE of $(\pi / 4)\hat P_Y^{(1)}(\epsilon; T,0)$, which obviously does not depend on $m$. The red horizontal line (0.22\%) corresponds to the MAPE of $\hat P_Y^{(2)}(\epsilon, m_Y^T; T,0)$, with $m_Y^T$ as in~\eqref{eqn:mXT}. Panel (c): the values of $m_Y^T$ computed from the 1000 independent replications of $Y$.}
    \label{fig:m_005}
\end{figure}

\begin{figure}
    \centering
    \includegraphics[width=\linewidth]{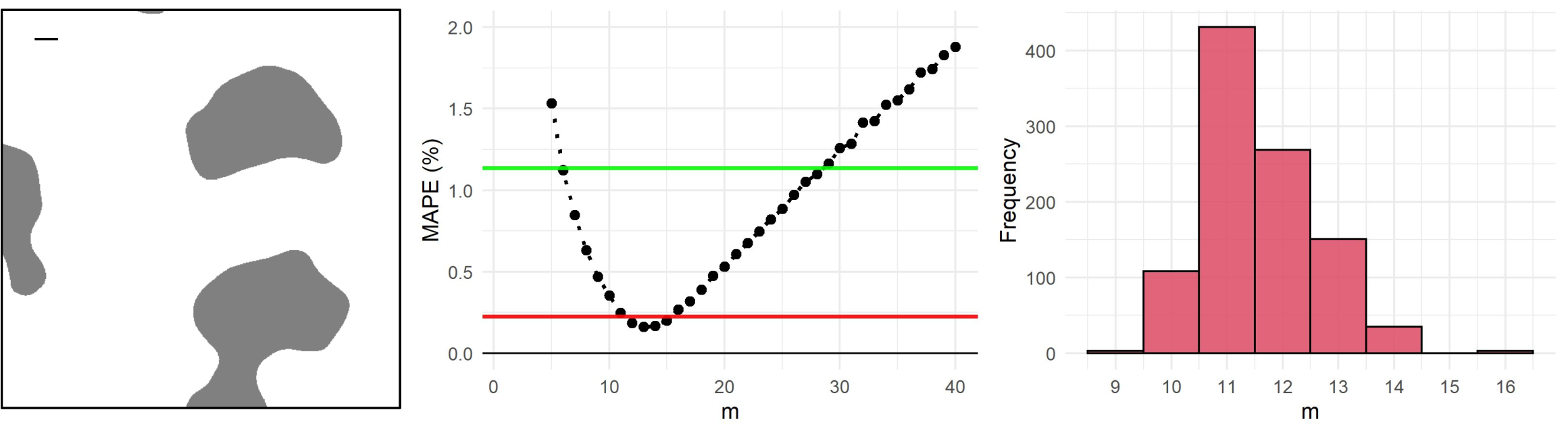}
    \put(-360,-11){(a)}
    \put(-224,-11){(b)}
    \put(-73,-11){(c)}
    \caption{See the caption of Figure~\ref{fig:m_005} for a description of each panel. In this case, $T = [-2.5,2.5]^2$ and $\epsilon = 5/511$. The MAPE of $(\pi / 4)\hat P_Y^{(1)}(\epsilon; T,0)$ is 1.13\%, and that of $\hat P_Y^{(2)}(\epsilon, m_Y^T; T,0)$ is 0.22\%.}
    \label{fig:m_02_60}
\end{figure}

\subsection{Behaviour of the perimeter estimator as a function of the level $u$}
Differently from our previous numerical studies, we illustrate the behaviour of $\hat P^{(2)}_Y(\epsilon,m_Y^T; T,u)$ as a function of the level $u$ in Figure~\ref{fig:levels_iso}, where $Y$ is the isotropic random field in Example~\ref{exa:general_X}. The same is done for an anisotropic field $X$ in Figure~\ref{fig:levels_aniso}.

\begin{figure}
    \centering
    \includegraphics[width=0.75\linewidth]{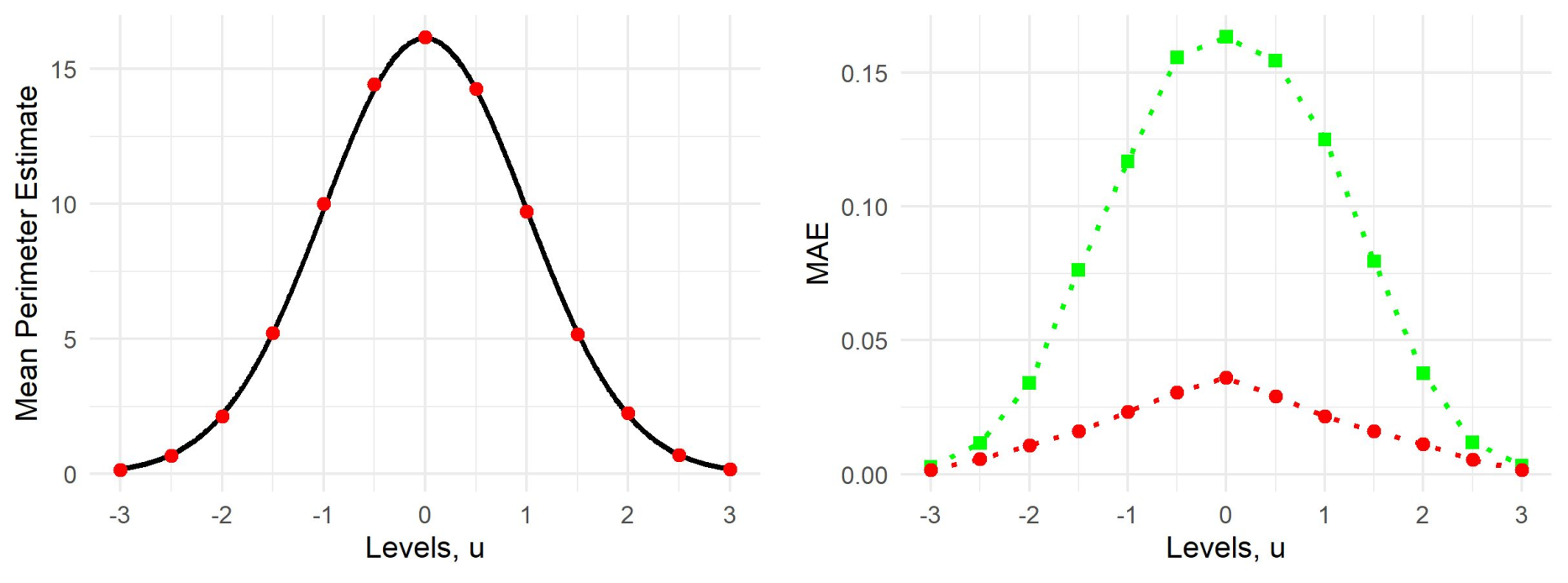}
    \put(-226,-12){(a)}
    \put(-72,-12){(b)}
    \caption{Illustration of perimeter estimation for several levels $u$. The stationary, isotropic, Gaussian random field $Y$ in Example~\ref{exa:general_X} is considered on $T=[-2.5,2.5]^2$ with a discretization of $\epsilon = 5/511$. Panel (a): the sample mean of 500 independent replications of $\hat P^{(2)}_Y(\epsilon,m_Y^T;T,u)$ plotted in red for several values of $u$, shown against $\E[P_Y^T(u)]$ in black (computed via the Gaussian Kinematic Formula in \citet[Theorem~15.9.5]{adler2007}). Panel (b): the MAE of the approximation of $\tilde{P}_{Y}(\epsilon;T,u)$ by $\hat P^{(2)}_Y(\epsilon,m_Y^T;T,u)$ (red circles) and $(\pi / 4)\hat P^{(1)}_Y(\epsilon,;T,u)$ (green squares).
}
    \label{fig:levels_iso}
\end{figure}

\begin{figure}
    \centering
    \includegraphics[width=1.03\linewidth]{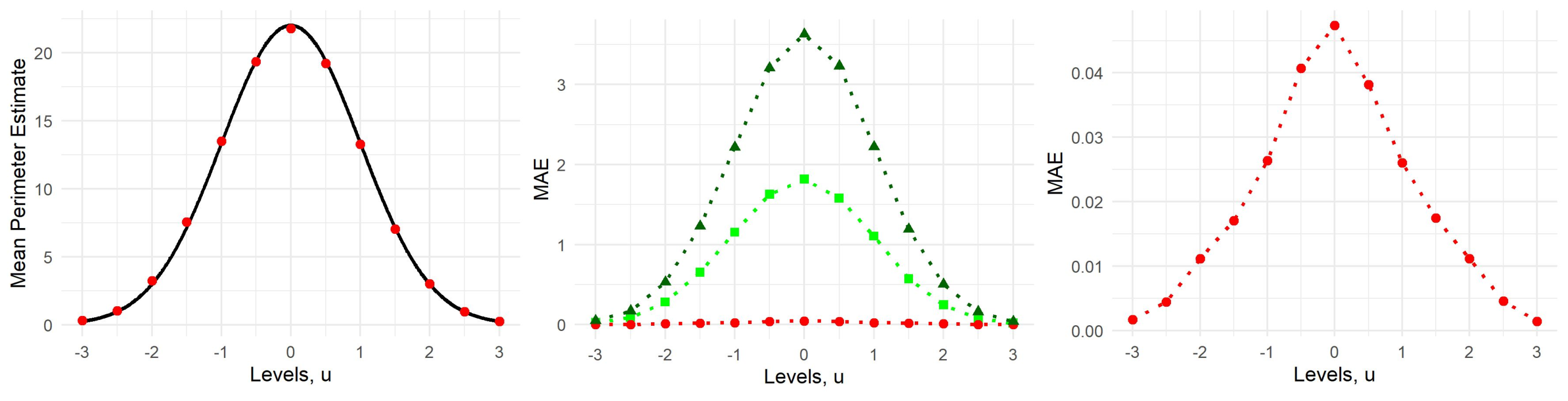}
    \put(-344,-11){(a)}
    \put(-209,-11){(b)}
    \put(-68,-11){(c)}
    \caption{The same experiment as depicted in Figure~\ref{fig:levels_iso}, but using the stationary, anisotropic, Gaussian random field $X(\cdot;2,0.5,0)$ in Example~\ref{exa:general_X}. Panel (a): $\E[P_X^T(u)]$, shown in black, is calculated via equation~\eqref{eqn:ellipse} and the Gaussian Kinematic Formula in \citet[Theorem~15.9.5]{adler2007}.  Panel (b): the MAE of the approximation of $\tilde{P}_{X}(\epsilon;T,u)$ by $\hat P^{(2)}_X(\epsilon,m_X^T;T,u)$ (red circles), $(\pi / 4)\hat P^{(1)}_X(\epsilon,;T,u)$ (green squares), and $\hat P^{(1)}_X(\epsilon,;T,u)$ (dark green triangles).  Panel (c): the MAE associated to $\hat P^{(2)}_X(\epsilon,m_X^T;T,u)$ shown again on a more appropriate y-axis scale.}
    \label{fig:levels_aniso}
\end{figure}

\section{Proofs}\label{sec:proofs}

This section provides detailed justifications for the theoretical results stated thus far. The following definition is used throughout this section.

\begin{defi}
For $s \in \R^2$, define the set $B_{s}^{(l)} := [0,l)^2 + s$, where $``+"$ in this context denotes the Minkowski sum. Let $\epsilon > 0$ and $m\in\N^+$. Define
$$\mathcal{V}_X^T(\epsilon,m; u) := \{s_{i,j}\in\mathcal{G}^{(T,\epsilon)}: i,j\in I^{(T,\epsilon,m)},\ B_{s_{i,j}}^{(m\epsilon)}\cap E^\partial_X(T,u)\neq \emptyset\}.$$
\end{defi}

The following lemma allows us to bound $\#\big(\mathcal{V}_X^T(\epsilon,m; u)\big)$, which amounts to an upper bound on the number of nonzero terms in the sum given by equation~\eqref{eqn:phat}. See Figure~\ref{fig:intersecting_boxes} in the appendix for an illustration that complements Lemma~\ref{lem:intersecting_boxes_bound}.

\begin{lemm}\label{lem:intersecting_boxes_bound}
Let $X$ be a random field satisfying Assumption~\ref{ass:basic}. For any $\epsilon>0$ and $m\in\N^+$,
$$\#\big(\mathcal{V}_X^T(\epsilon, m; u)\big) \leq 4\Big(\frac{P_X^T(u)}{m\epsilon} + \#\big(\Gamma_X^T(u)\big)\Big),\as$$
\end{lemm}

\begin{proof}
The squares of side length $m\epsilon$ in the set $\mathcal{B} := \{B_{s_{i,j}}^{(m\epsilon)}:i,j\in I^{(T,\epsilon,m)}\}$ are disjoint and cover $T$.
For each $\gamma\in\Gamma_X^T(u)$, it is possible to find connected subsets of $\gamma$, namely $\beta_{\gamma,1}$, \ldots, $\beta_{\gamma,M_\gamma}$, that satisfy
$$\gamma = \bigcup_{i=1}^{M_\gamma}\beta_{\gamma,i},$$
where
$$M_\gamma := \Big\lfloor\frac{\mathcal{H}^1(\gamma)}{m\epsilon}\Big\rfloor + 1,$$
and for all $i\in\{1,\ldots,M_{\gamma}\}$,
$$\mathcal{H}^1(\beta_{\gamma, i}) \leq m\epsilon.$$
Each $\beta_{\gamma,i}$ can intersect at most 4 elements of $\mathcal{B}$. Since
$$ E^\partial_X(T,u) = \bigcup_{\gamma\in\Gamma_X^T(u)}\ \bigcup_{i=1}^{M_\gamma}\beta_{\gamma,i},$$
it follows that
\begin{align*}
    \#\big(\mathcal{V}_X^T(\epsilon, m; u)\big) &= \#\big(\{b\in\mathcal{B}:b \cap E^\partial_X(T,u) \neq \emptyset\}\big)\\
    &\leq 4\sum_{\gamma\in\Gamma_X^T(u)}M_\gamma \leq 4\Big(\frac{P_X^T(u)}{m\epsilon} + \#\big(\Gamma_X^T(u)\big)\Big),\as
\end{align*}
\end{proof}

\begin{proof}[Proof of Theorem~\ref{thm:consistent_Phat}]
Let $\omega\in\Omega$ be such that $\Lambda_{X(\omega)}^T(u)$, defined in Definition~\ref{def:resolution}, is positive (note that almost any $\omega\in\Omega$ will suffice, as discussed in Remark~\ref{rem:positive_as}). There exists $n_0\in\N^+$ such that $E_{X(\omega)}(u)$ is resolved by $m_n\epsilon_n$ in $T$ for all $n\geq n_0$ (see Definition~\ref{def:resolution}). Fix $s_{i,j}\in \mathcal{V}_{X(\omega)}^T(\epsilon_n,m_n;u)$ and $n\geq n_0$.
Let $\gamma := B_{s_{i,j}}^{(m_n\epsilon_n)}\cap E^\partial_{X(\omega)}(T,u)$.
It follows from our construction that $\mathcal{Y}_{X(\omega)}^T(u)\cap \gamma$ contains at most one element, since the spacing between points in $\mathcal{Y}_{X(\omega)}^T(u)$ is larger than the diameter of $B_{s_{i,j}}^{(m_n\epsilon_n)}$.
It also follows from our construction that $\gamma$ is either connected, or the union of two maximally connected subsets. To see this, note that the planar curvature of $\gamma$ does not exceed $1/(m_{n_0}\epsilon_{n_0})$ since $m_{n_0}\epsilon_{n_0}$ is smaller than the reach of both $E_{X(\omega)}(T,u)$ and $T\setminus E_{X(\omega)}(u)$.
Therefore, the curve is bounded by the planar arcs of radius $m_{n_0}\epsilon_{n_0}$ as shown in Figure~\ref{fig:gammas} \citep{dubins1961}.
We aim to bound the absolute difference between the length of $\gamma$ and its contribution to $\hat{P}^{(2)}_{X(\omega)}(\epsilon_n, m_n; T, u)$. To this end, the two cases shown in Figure~\ref{fig:gammas} are considered separately.

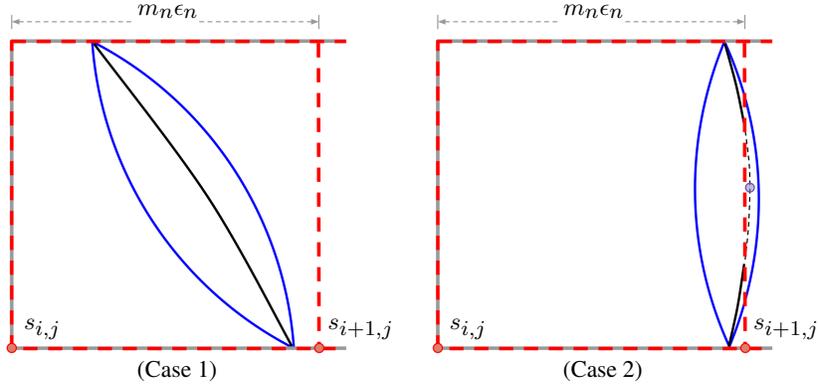
\begin{figure}
    \centering
    \includegraphics[width=0.74\linewidth]{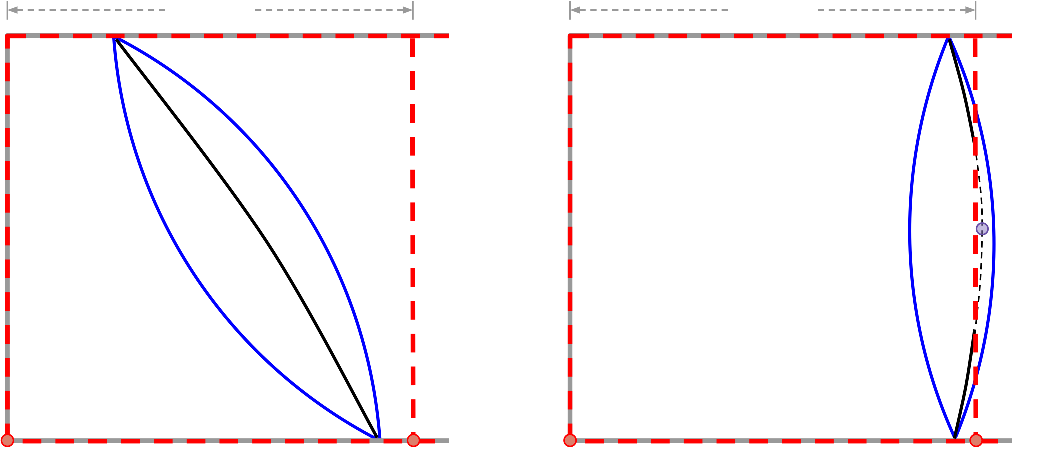}
    \put(-251,-8){(Case 1)}
    \put(-294,10){$s_{i,j}$}
    \put(-179,10){$s_{i+1,j}$}
    \put(-251,129){$m_n\epsilon_n$}
    \put(-90,-8){(Case 2)}
    \put(-133,10){$s_{i,j}$}
    \put(-18,10){$s_{i+1,j}$}
    \put(-90,129){$m_n\epsilon_n$}
    \caption{(Case 1) The curve $\gamma$ shown in black is bounded by the planar arcs of radius $m_{n_0}\epsilon_{n_0}$ shown in blue. (Case 2) Here, $\gamma$ shown in black is not connected, and the only point in $\mathcal{Y}_{X(\omega)}^T(u)\cap B_{s_{i+1,j}}^{(m_n\epsilon_n)}$ is highlighted in purple.}
    \label{fig:gammas}
\end{figure}

\begin{description}
\item[Case 1:] \emph{The curve $\gamma$ is connected (see the left panel of Figure~\ref{fig:gammas})}.
The closure of $\gamma$ can be parametrized by a continuous injective vector function $\textbf{x} : [0,1] \rightarrow \R^2$. For $\alpha \in [0,1]$, define
\begin{equation}\label{eqn:TV}
\mathrm{TV}_k(\alpha;{s_{i,j}}) := \int_0^\alpha|x_k'(s)|\ \d s,\qquad k \in\{1,2\},
\end{equation}
so that $\mathrm{TV}_k(1;{s_{i,j}})$ corresponds to the total variation of $\gamma$ in the $k^{\mathrm{th}}$ principle Cartesian direction of $\R^2$.
As a consequence of the coarea formula \cite[Equation~(7.4.15)]{adler2007}, the quantity $\epsilon_n N_{X(\omega),h}(i,j;u)$ (see Definition~\ref{def:estimator}) is a Riemann sum that approximates the definite integral $\mathrm{TV}_1(1; {s_{i,j}})$. The total error can therefore be bounded above by
\begin{equation}\label{eqn:coarea_approx}
    \big|\epsilon_n N_{X(\omega),h}(i,j;u) - \mathrm{TV}_1(1;s_{i,j})\big| \leq 4\epsilon_n,
\end{equation}
as suggested by Figure~\ref{fig:coarea}, found in the appendix.
Analogously,
\begin{equation*}
    \big|\epsilon_n N_{X(\omega),v}(i,j;u) - \mathrm{TV}_2(1;s_{i,j})\big| \leq 4\epsilon_n.
\end{equation*}
Let
\begin{equation}\label{eqn:lnhat}
    \hat{l}_n(s_{i,j}) := \epsilon_n\big|\big|\big(N_{X(\omega),v}(i,j;u), N_{X(\omega),h}(i,j;u)\big)\big|\big|_2,
\end{equation}
and we achieve the following bound by the triangle inequality
\begin{equation}\label{eqn:lnhat_closeto_TV}
    \Big|\hat{l}_n(s_{i,j}) - \big|\big|\big(\mathrm{TV}_1(1;s_{i,j}),\mathrm{TV}_2(1;s_{i,j})\big)\big|\big|_2\Big| \leq 4\sqrt{2}\epsilon_n.
\end{equation}
It is clear that
\begin{equation}\label{eqn:line_seg_leq_TV}
    ||\textbf{x}(1)-\textbf{x}(0)||_2 \leq \big|\big|\big(\mathrm{TV}_1(1;s_{i,j}),\mathrm{TV}_2(1;s_{i,j})\big)\big|\big|_2,
\end{equation}
since the computation of the left-hand side of Equation ~\eqref{eqn:line_seg_leq_TV} involves the same integral as in~\eqref{eqn:TV} but without the absolute values. In addition, let
$$l(\alpha;s_{i,j}) := \int_0^\alpha||\textbf{x}'(s)||_2\ \d s$$
denote the length of $\textbf{x}(s)$ for $s\in[0,\alpha]$.
It follows from the definition of the derivative and the reverse triangle inequality that for all $\alpha\in(0,1)$,
\begin{equation*}
    \bigg|\frac{\partial}{\partial \alpha} \big|\big|\big(\mathrm{TV}_1(\alpha;s_{i,j}),\mathrm{TV}_2(\alpha;s_{i,j})\big)\big|\big|_2\bigg|
    \leq \big|\big|\big(x_1'(\alpha), x_2'(\alpha)\big)\big|\big|_2= \frac{\partial}{\partial \alpha} l(\alpha;s_{i,j}).
\end{equation*}
Therefore,
\begin{equation}\label{eqn:TV_leq_length}
    \big|\big|\big(\mathrm{TV}_1(1;s_{i,j}),\mathrm{TV}_2(1;s_{i,j})\big)\big|\big|_2 \leq l(1;s_{i,j}).
\end{equation}
Since the curvature of $\gamma$ is bounded above by the inverse of $\Lambda^T_{X(\omega)}(u)$, we apply a well known result from Schwartz \citep{dubins1961} that guarantees that \begin{equation}\label{eqn:length_leq_a}
    l(1;s_{i,j}) \leq a(s_{i,j}),
\end{equation}
where $a(s_{i,j})$ is the length of the smallest planar arc with radius $m_{n_0}\epsilon_{n_0}$ that has endpoints $\textbf{x}(0)$ and $\textbf{x}(1)$.
The Taylor expansion of the sine function shows the existence of $K\in\R^+$ independent of $s_{i,j}$ and $n$ such that
\begin{equation}\label{eqn:Taylor_sine}
    \Big|a(s_{i,j})-||\textbf{x}(1) - \textbf{x}(0)||_2\Big| \leq K||\textbf{x}(1) - \textbf{x}(0)||_2^3 \leq K(\sqrt{2}m_n\epsilon_n)^3.
\end{equation}
Assembling the bounds demonstrated in Equations~\eqref{eqn:line_seg_leq_TV},~\eqref{eqn:TV_leq_length}, and~\eqref{eqn:length_leq_a}, we get
$$||\textbf{x}(1)-\textbf{x}(0)||_2 \leq \big|\big|\big(\mathrm{TV}_1(1;s_{i,j}),\mathrm{TV}_2(1;s_{i,j})\big)\big|\big|_2 \leq
l(1;s_{i,j})\leq
a(s_{i,j}),$$
which in combination with~\eqref{eqn:Taylor_sine} implies
\begin{equation}\label{eqn:length_closeto_TV}
    \Big|l(1;s_{i,j}) - \big|\big|\big(\mathrm{TV}_1(1;s_{i,j}),\mathrm{TV}_2(1;s_{i,j})\big)\big|\big|_2\Big| \leq K(\sqrt{2}m_n\epsilon_n)^3.
\end{equation}
Now, combining Equations~\eqref{eqn:length_closeto_TV} and~\eqref{eqn:lnhat_closeto_TV} by the triangle inequality yields
\begin{equation}\label{eqn:case1_final}
    \Big|\hat{l}_n(s_{i,j}) - l(1;s_{i,j})\Big| \leq K(\sqrt{2}m_n\epsilon_n)^3 + 4\sqrt{2}\epsilon_n.
\end{equation}

\item[Case 2:] \emph{The curve $\gamma$ has two connected components (see the right panel of Figure~\ref{fig:gammas})}. Similarly to Case 1, we parametrize the closure of each maximally connected subset of $\gamma$ with continuous injective vector functions $\textbf{x} : [0,1] \rightarrow \R^2$ and $\textbf{y} : [0,1] \rightarrow \R^2$. For $\alpha\in [0,1]$, define
\begin{equation*}
\mathrm{TV}_k(\alpha;{s_{i,j}}) := \int_0^\alpha\big(|x_k'(s)|+|y_k'(s)|\big) \ \d s,\qquad k \in\{1,2\}.
\end{equation*}
With $\hat{l}_n(s_{i,j})$ defined as in~\eqref{eqn:lnhat}, equation~\eqref{eqn:lnhat_closeto_TV} holds.
Now, consider the curve $\tilde\gamma := (B_{s_{i,j}}^{(m_n\epsilon_n)} \cup B_{s_{i+1,j}}^{(m_n\epsilon_n)})\cap E^\partial_{X(\omega)}(T,u)$, which is $\gamma$ in union with the middle section in the adjacent box $B_{s_{i+1,j}}^{(m_n\epsilon_n)}$ (where we have assumed, without loss of generality, that the ``middle section'' is in the box to the right). It is clear that $\tilde\gamma$ is connected, so its closure can be parametrized by the continuous injective vector function $\textbf{z} : [0,1] \rightarrow \R^2$. Define
\begin{equation*}
\widetilde{\mathrm{TV}_k}(\alpha;{s_{i,j}}) := \int_0^\alpha|z_k'(s)| \ \d s,\qquad k \in\{1,2\}
\end{equation*}
and
$$\tilde{l}(\alpha;s_{i,j}) := \int_0^\alpha||\textbf{z}'(s)||_2\ \d s.$$
By the same arguments that led to equation~\eqref{eqn:length_closeto_TV}, it holds that
\begin{equation}\label{eqn:tildes_close}
    0 \leq \tilde{l}(1;s_{i,j}) - \big|\big|\big(\widetilde{\mathrm{TV}_1}(1;s_{i,j}),\widetilde{\mathrm{TV}_2}(1;s_{i,j})\big)\big|\big|_2 \leq K(\sqrt{2}m_n\epsilon_n)^3,
\end{equation}
where $K\in\R^+$ is independent of $s_{i,j}$ and $n$.
Let
$${l}(1;s_{i,j}) := \int_0^1\big(||\textbf{x}'(s)||_2+||\textbf{y}'(s)||_2\big)\ \d s$$
be the total length of $\gamma$. Then
$\tilde{l}(1;s_{i,j}) = l(1;s_{i,j}) + l(1;s_{i+1,j}),$
and
\begin{align*}
    \big|\big|\big(\widetilde{\mathrm{TV}_1}(1;s_{i,j}),\widetilde{\mathrm{TV}_2}(1;s_{i,j})\big)\big|\big|_2 \leq &\
    \big|\big|\big(\mathrm{TV}_1(1;s_{i,j}),\mathrm{TV}_2(1;s_{i,j})\big)\big|\big|_2\\
    &+ \big|\big|\big(\mathrm{TV}_1(1;s_{i+1,j}),\mathrm{TV}_2(1;s_{i+1,j})\big)\big|\big|_2
\end{align*}
by the triangle inequality. Therefore,~\eqref{eqn:tildes_close} can be written as
\begin{align}
     \Big(l(1;s_{i,j}) &- \big|\big|\big(\mathrm{TV}_1(1;s_{i,j}),\mathrm{TV}_2(1;s_{i,j})\big)\big|\big|_2\Big)
     +\nonumber\\
     \Big(l(1;s_{i+1,j}) &-
     \big|\big|\big(\mathrm{TV}_1(1;s_{i+1,j}),\mathrm{TV}_2(1;s_{i+1,j})\big)\big|\big|_2\Big) \leq K(\sqrt{2}m_n\epsilon_n)^3.\label{eqn:case2_final}
\end{align}
By the arguments in Case 1 that led to equation~\eqref{eqn:TV_leq_length}, it follows that $$l(1;s_{i+1,j}) \geq \big|\big|\big(\mathrm{TV}_1(1;s_{i+1,j}),\mathrm{TV}_2(1;s_{i+1,j})\big)\big|\big|_2,$$ and by the same arguments, $$l(1;s_{i,j}) \geq \big|\big|\big(\mathrm{TV}_1(1;s_{i,j}),\mathrm{TV}_2(1;s_{i,j})\big)\big|\big|_2.$$ Therefore, both~\eqref{eqn:length_closeto_TV} and~\eqref{eqn:case1_final} follow from equation~\eqref{eqn:case2_final}.
\end{description}
Following from equation~\eqref{eqn:case1_final}, we have
\begin{align*}
    \big|\hat{P}^{(2)}_{X(\omega)}(\epsilon_n, m_n; T, u) - {P}_{X(\omega)}^T(u)\big|
    &= \bigg|\sum_{s_{i,j}\in \mathcal{V}_{X(\omega)}^T(\epsilon_n, m_n; u)}
    \Big(\hat{l}_n(s_{i,j}) - l(1;s_{i,j})\Big) \bigg|\\
    &\leq \sum_{s_{i,j}\in \mathcal{V}_{X(\omega)}^T(\epsilon_n, m_n; u)}\big|\hat{l}_n(s_{i,j}) - l(1;s_{i,j})\big|\\
    &\leq \#\big(\mathcal{V}_{X(\omega)}^T(\epsilon_n, m_n; u)\big)\ 2\sqrt{2}(Km_n^3\epsilon_n^3 + 2\epsilon_n).
\end{align*}

By Lemma~\ref{lem:intersecting_boxes_bound},
\begin{align}
    g_n\big|\hat{P}^{(2)}_{X(\omega)}&(\epsilon_n, m_n; T, u)  - {P}_{X(\omega)}^T(u)\big|\nonumber\\
    &\leq 8\sqrt{2}g_n
    \bigg(\frac{P_{X(\omega)}^T(u)}{m_n\epsilon_n} + \#\big(\Gamma_{X(\omega)}^T(u)\big)\bigg)\Big(Km_n^3\epsilon_n^3 + 2\epsilon_n\Big)\nonumber\\
    & = 8\sqrt{2}\ \frac{g_n}{m_n}
    \Big(P_{X(\omega)}^T(u) + m_n\epsilon_n\#\big(\Gamma_{X(\omega)}^T(u)\big)\Big)\Big(Km_n^3\epsilon_n^2 + 2\Big),\label{eqn:penultimate_thm1}
\end{align}
which tends to 0 as $n\rightarrow\infty$. This convergence holds for almost every $\omega\in\Omega$, since $\Lambda_{X}^T(u)$ is almost surely positive.
\end{proof}

\begin{proof}[Proof of Corollary~\ref{cor:as_convergence}]
The last expression in equation~\eqref{eqn:penultimate_thm1} tends to 0 under the relaxed constraint on $(\epsilon_n)_{n\geq 1}$ if $g_n\equiv 1$ for all $n\in\N^+$.
\end{proof}

\begin{proof}[Proof of Proposition~\ref{prp:L1}]
If a sequence is uniformly integrable, convergence in $L^1(\Omega)$ is equivalent to convergence in probability. Therefore, by Corollary~\ref{cor:as_convergence}, it suffices to show that $\big(\hat{P}^{(2)}_X(\epsilon_n, m_n; T, u)\big)_{n\geq 1}$ is bounded above by an element of $L^1(\Omega)$ uniformly in $n$. Note that for each $n\geq 1$,
$$\hat{P}^{(2)}_X(\epsilon_n, m_n; T, u) \leq \hat{P}^{(1)}_X(\epsilon_n; T, u),\as$$
since the 2-norm is inferior to the 1-norm.
Now, consider the quantity
$$G_n := \#\big(\{s\in\mathcal{G}^{(T,\epsilon_n)}:B_{s}^{(\epsilon_n)}\cap E^\partial_{X}(T,u)\neq \emptyset\}\big),$$
which represents the number of pixels of side length $\epsilon_n$ that the curve $E^\partial_{X}(T,u)$ intersects.
Almost surely, $\hat{P}^{(1)}_X(\epsilon_n; T, u)$ is at most $4\epsilon_n$ (the perimeter of one pixel) times $G_n$. By the same arguments used to prove Lemma~\ref{lem:intersecting_boxes_bound}, we have for all $n\geq 1$,
$$G_n \leq 4\Big(\frac{P_X^T(u)}{\epsilon_n}+\#\big(\Gamma_X^T(u)\big)\Big),\as$$
and
$$\hat{P}^{(2)}_X(\epsilon_n, m_n; T, u) \leq \hat{P}^{(1)}_X(\epsilon_n; T, u) \leq 4\epsilon_nG_n \leq 16\Big(P_X^T(u) + \sup_n(\epsilon_n)\#\big(\Gamma_X^T(u)\big)\Big),\as$$
which is in $L^1(\Omega)$ by Assumption~\ref{ass:integrable_quantities}.
\end{proof}

\begin{proof}[Proof of Proposition~\ref{prp:Tn}]
Let
$$W_n:= \frac{\hat{P}^{(2)}_{X}(\epsilon_{n}, m_{n}; T_n, u) - {P}_{X}^{T_n}(u)}{\sqrt{\nu(T_n)}}.$$
Given that $E_X(u)$ is resolved by $m_n\epsilon_n$ in $T_n$ for fixed $n\in\N^+$, equation~\eqref{eqn:penultimate_thm1} holds with $g_n=1/\sqrt{\nu(T_n)}$, implying
\begin{align}\label{eqn:W_n_bound}
|W_n| &\leq \frac{8}{m_n}\sqrt{\frac{2}{\nu(T_n)}}\Big(P_{X}^{T_n}(u) + m_n\epsilon_n\#\big(\Gamma_{X}^{T_n}(u)\big)\Big)\Big(Km_n^3\epsilon_n^2 + 2\Big)\nonumber\\
&= \frac{8\sqrt{2\nu(T_n)}}{m_n}\bigg(\frac{P_{X}^{T_n}(u)}{\nu(T_n)} + m_n\epsilon_n\frac{\#\big(\Gamma_{X}^{T_n}(u)\big)}{\nu(T_n)}\bigg)\Big(Km_n^3\epsilon_n^2 + 2\Big),
\end{align}
where $K\in\R^+$ is independent of $n$. Note that for any $n\in\N^+$,
$$T_n = \bigcup_{i=1}^{n^2}T_n^{(i)},$$
for a family of sets $(T_n^{(i)})_{i=1,\ldots,n^2}$, each of which being congruent to $T_1$. Then
$$\E\bigg[\frac{P_X^{T_n}(u)}{\nu(T_n)}\bigg] = \E\bigg[\frac{\sum_{i=1}^{n^2}P_X^{T_n^{(i)}}(u)}{n^2\nu(T_1)}\bigg] = \E\bigg[\frac{P_X^{T_1}(u)}{\nu(T_1)}\bigg] < \infty $$
and
$$\E\bigg
[\frac{\#\big(\Gamma_X^{T_n}(u)\big)}{\nu(T_n)}\bigg] \leq
\E\bigg[\frac{\sum_{i=1}^{n^2}\#\big(\Gamma_X^{T_n^{(i)}}(u)\big)}{n^2\nu(T_1)}\bigg] = \E\bigg[\frac{\#\big(\Gamma_X^{T_1}(u)\big)}{\nu(T_1)}\bigg] < \infty,$$
by Assumption~\ref{ass:integrable_quantities}.
This implies that both
\begin{equation*}
\limsup_{n\rightarrow\infty}\frac{P_X^{T_n}(u)}{\nu(T_n)}\qquad \mathrm{and}\qquad \limsup_{n\rightarrow\infty}\frac{\#\big(\Gamma_X^{T_n}(u)\big)}{\nu(T_n)}
\end{equation*}
are finite almost surely.
Therefore, the final expression in~\eqref{eqn:W_n_bound} tends to 0 almost surely, since $\sqrt{\nu(T_n)}/m_n\rightarrow 0$ by assumption. Now, denote the random event $A_n:=\{m_n\epsilon_n < \Lambda_X^{T_n}(u)\}$, and let $A_n^C$ denote its complement. Since $\P(A_n)\rightarrow 1$ as $n\rightarrow\infty$ by assumption, it holds that for any $\eta > 0$,
$$\P(|W_n| > \eta) \leq \P(|W_n| > \eta\ |\ A_n)\P(A_n) + \P(A_n^C) \rightarrow 0$$
as $n\rightarrow\infty$.
\end{proof}

\begin{proof}[Proof of Theorem~\ref{thm:CLT}]
The Central Limit Theorem in \cite{iribarren1989} for $P_X^{T_n}(u)$ at the fixed level $u\in\R$ is implied by the constraints on $X$. The result is proven for a single level $u$, but as noted in the Discussion of \cite{kratz2018} and in \cite{shashkin2013}, the Cram{\'e}r-Wald device can be used to extend the arguments to the multivariate setting.

The Central Limit Theorem for the perimeter is then written as follows. For any $\textbf{u}\in\R^k$ satisfying the given constraints, it holds that
\begin{equation}\label{eqn:kratz_CLT}
\frac{P_X^{T_n}(\textbf{u}) - \E[P_X^{T_n}(\textbf{u})]}{\sqrt{\nu(T_n)}} \stackrel{\d}{\longrightarrow} \mathcal{N}_k\big(\textbf{0},\Sigma(\textbf{u})\big),\qquad n\rightarrow\infty.
\end{equation}
Equation~\eqref{eqn:CLT_statement} is obtained by combining equation~\eqref{eqn:kratz_CLT}, Proposition~\ref{prp:Tn}, and Slutsky's theorem.

By writing $P_X^T(u) = \lim_{\epsilon\to 0}1/(2\epsilon)\int_T \I{\vert X(s)-u\vert < \epsilon}||\nabla X(s)||_2\, \d s$, (see, for instance, Proposition~6.13 in \citet{azais2009}), it is easily checked that for $u_1,u_2\in\R$,
$$\E[P_X^T(u_1)] = \nu(T) f(u_1)\E\big[||\nabla X(s)||_2\ \big|\ X(s) = u_1\big]$$
and
\begin{align*}
\E[P_X^T(u_1)P_X^T(u_2)] =& \int_{T}\int_{T} g_{s_2-s_1}(u_1,u_2)\\
&\times \E\big[||\nabla X(s_1)||_2||\nabla X(s_2)||_2\ \big|\ X(s_1) = u_1, X(s_2) = u_2\big]\d s_1\d s_2,
\end{align*}
where $f$ denotes the marginal density function of $X$, and $g_s$, the joint density function of $\big(X(0), X(s)\big)$. Hence the result in \eqref{eqn:var_Sigma}.
\end{proof}

\begin{proof}[Proof of Corollary~\ref{cor:gaussian}]
Under the given constraints, it is clear that Assumption~\ref{ass:basic} is satisfied. Also following from the hypotheses, the gradient of $X$ and the Hessian matrix of $X$ are independent with Gaussian entries, and thus the conditions of Theorem~11.3.3 of \cite{adler2007} are satisfied.
Therefore, $X$ is almost surely \textit{suitably regular} \cite[Definition~6.2.1]{adler2007} over bounded rectangles, which implies the conditions of Assumption~\ref{ass:weak}.
The expectations $\E\big[P_X^{T_n}(u)\big]$ and $\E\big[\#\big(\Gamma_X^{T_n}(u)\big)\big]$ are shown to be finite in \citet[Theorem~13.2.1]{adler2007} and \cite{beliaev2020} respectively, implying the conditions of Assumption~\ref{ass:integrable_quantities}. Therefore, Proposition~\ref{prp:Tn} holds, which in combination with the Central Limit Theorem in \citet[Theorem~4.7]{berzin2021} yields the result.
\end{proof}

\bigskip
\section*{Discussion} We have shown for a large class of random fields that $\hat{P}^{(p)}_{X}(\epsilon, m; T, u)$ with $p=2$ is a consistent and asymptotically normal estimator for $P_X^T(u)$. Our numerous simulation studies showcase the various cases where it is advantageous to use the norm $p=2$ as opposed to $p=1$. An obvious example is when $X$ is not known to be isotropic. For $p > 2$, we do not expect $\hat{P}^{(p)}_{X}(\epsilon, m; T, u)$ to have desirable properties, since there is a bias introduced for certain orientations of the curve $E_X^{\partial}(T,u)$. There is a natural extension of $\hat{P}^{(p)}_{X}(\epsilon, m; T, u)$ to random fields defined on $\R^d$, with $d>2$, and it is plausible that analogous results hold in this multivariate setting. Results such as the central limit theorems in \cite{shashkin2013}, \cite{mueller2017}, and \cite{kratz2018}, which hold in arbitrary dimension, will be useful to study the Gaussian fluctuations of our estimate.

Future work might also investigate the rate at which $\Lambda_X^T(u)$ tends weakly to 0 as $T \nearrow \R^2$, which would provide a more explicit constraint on the rate at which $\epsilon_n\rightarrow 0$ in Proposition~\ref{prp:Tn}.

Furthermore, we plan to study how the proposed  perimeter estimate can be used to build a test statistics for isotropy testing based on the length of level curves of smooth random fields. This future analyse could enrich the existing literature of isotropy testing based on functionals of level curves \citep{wschebor1985, cabana1987, Fou17, berzin2021}.\smallskip

The proposed estimator works with observations available at a set of locations forming a regular grid. A large variety of datasets possess this format, such as outputs of various types of models (\emph{e.g.}, climate, hydrology), remote sensing data, or imaging data (\emph{e.g.}, in medecine). However,  geostatistical spatial data are sometimes not observed on regular grids, such as meteorological data observed over a network of weather stations not organised in any grid structures. In such cases, one could first apply a deterministic or stochastic interpolation method (\emph{e.g.}, bilinear interpolation, geostatistical kriging) to pre-process data to make them available on a  regular grid, and then use the grid-based estimator.

In this paper, we have focused on perimeter estimator properties in the case of a single replicate of the random field with one or several fixed levels $u$. Properties of estimators of Lipschitz--Killing curvatures, including the perimeter, could further be studied when the level $u$ tends towards the upper endpoint of the marginal distribution of $X$. This setting is relevant for extreme-value theory of stochastic processes \citep[][Chapters~9--10]{deHaan2006}. Jointly with decreasing pixel size and increasing domain $T$,  we would further have to control the rate at which the perimeter  tends towards zero  as $u$ increases, where ultimately the excursion set is almost surely empty. The combination of the results obtained for our perimeter estimator  with asymptotics of the exact perimeter for increasing level $u$ \citep{adler2007} could be useful to establish asymptotic results and appropriate estimators for the perimeter and for other excursion-set geometrical features  at extreme thresholds.

\bigskip

\begin{acks}[Acknowledgments]
The authors are grateful to Anne Estrade, C\'{e}line Duval and Hermine Bierm\'{e} for fruitful discussions. The authors sincerely express their gratitude to the anonymous referee and associated editor for their valuable comments and remarks which improve the quality of the present work. This work has been supported by the French government, through the 3IA C\^{o}te d'Azur Investments in the Future project managed by the National Research Agency (ANR) with the reference number ANR-19-P3IA-0002. This work has been partially supported  by the project  ANR MISTIC (ANR-19-CE40-0005).
\end{acks}

\bigskip

\bibliographystyle{imsart-nameyear.bst} 
\bibliography{biblio.bib}       

\vspace{10pt}
\textbf{Correspondence:} Ryan Cotsakis, Laboratoire J.A. Dieudonn\'e, Universit\'e C\^ote d’Azur, 28 avenue Valrose,
06108 Nice Cedex 02, France.\\
Email: \href{mailto:ryan.cotsakis@unice.fr}{ryan.cotsakis@unice.fr}

\newpage
\begin{appendix}
\section*{}
Here, we provide two figures; one to complement Lemma~\ref{lem:intersecting_boxes_bound}, and the other, equation~\eqref{eqn:coarea_approx}.

\begin{figure}[H]
    \centering
    \includegraphics[width=0.32\linewidth]{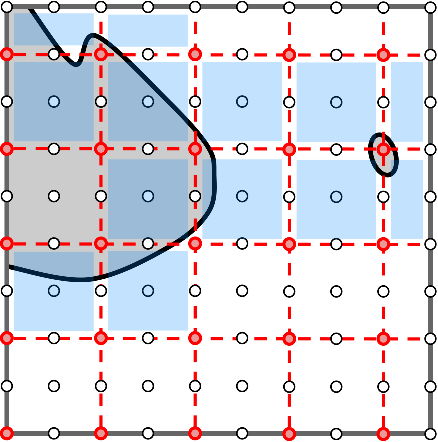}
    \caption{An illustration to aid Lemma~\ref{lem:intersecting_boxes_bound}. With $m=2$, the curve $E^\partial_X(T,u)$ shown in black intersects 13 elements of $\{B_{s_{i,j}}^{(m\epsilon)}:i,j\in I^{(T,\epsilon,m)}\}$, which are highlighted in blue. Thus, $\#\big(\mathcal{V}_X^T(\epsilon, m; u)\big) = 13$.}
    \label{fig:intersecting_boxes}
\end{figure}

\begin{figure}[H]
    \centering
    \includegraphics[width=0.42\linewidth]{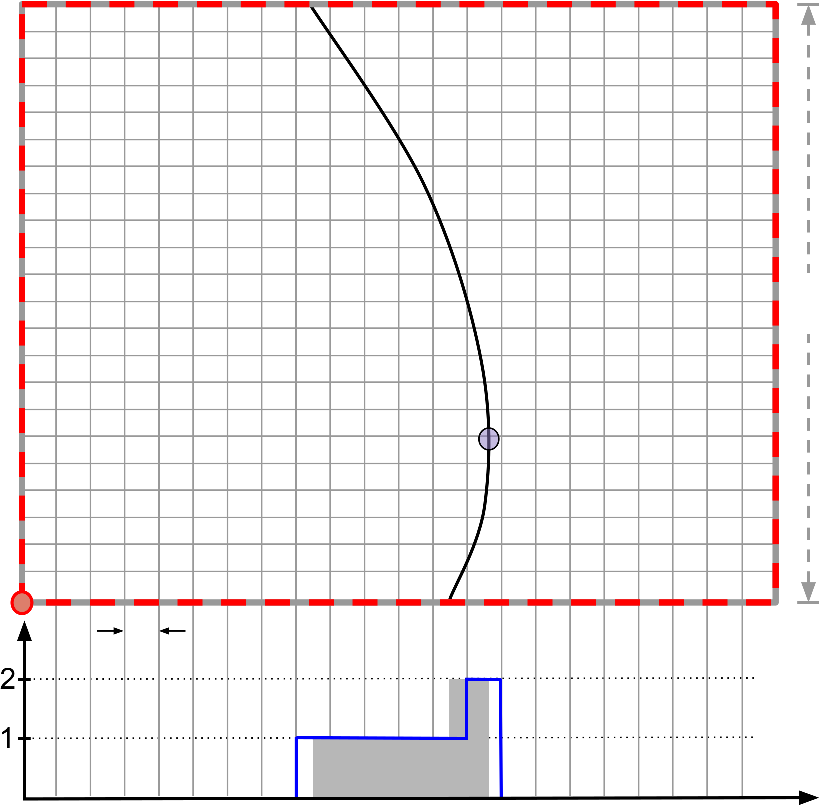}
    \put(-8,102){$m_n\epsilon_n$}
    \put(-145,30){$\epsilon_n$}
    \put(-163,48){$s_{i,j}$}
    \caption{The approximation of $TV_1(1,s_{i,j})$ in~\eqref{eqn:TV} by $\epsilon_n N_{X(\omega),h}(i,j;u)$ (see Definition~\ref{def:estimator}). The black curve $\gamma$ is shown in $B_{s_{i,j}}^{(m_n\epsilon_n)}$, which we outline in dashed red. The definite integral $TV_1(1,s_{i,j})$ is represented by the grey area, and is approximated by $\epsilon_n N_{X(\omega),h}(i,j;u) = 7\epsilon_n$, the area under the blue curve. The absolute error of this approximation is clearly bounded above by $4\epsilon_n$ as stated in equation~\eqref{eqn:coarea_approx}. Highlighted in purple is a point in $\mathcal{Y}_{X(\omega)}^T(u)$ (see equation~\eqref{eqn:curly_Y}).}
    \label{fig:coarea}
\end{figure}

\end{appendix}

\end{document}